\documentclass[12pt]{amsart}

\usepackage[utf8]{inputenc}
\usepackage{amsfonts}
\usepackage{amsthm}
\usepackage{amssymb}
\usepackage{amsmath}
\usepackage{amscd}
\usepackage{latexsym,dsfont}
\usepackage{bbm}
\usepackage{mathrsfs}

\usepackage{times}
\usepackage{microtype}
\usepackage{setspace}
\usepackage[margin=1.2in]{geometry}

\usepackage{cite}

\usepackage[colorlinks=true, pdfstartview=FitV, linkcolor=blue,
citecolor=blue, urlcolor=blue]{hyperref}

\newtheorem{theorem}{Theorem}[section]

\theoremstyle{definition}

\theoremstyle{remark}
\newtheorem{remark}[theorem]{Remark}

\numberwithin{equation}{section}
 \allowdisplaybreaks

\def\Xint#1{\mathchoice
   {\XXint\displaystyle\textstyle{#1}}%
   {\XXint\textstyle\scriptstyle{#1}}%
   {\XXint\scriptstyle\scriptscriptstyle{#1}}%
   {\XXint\scriptscriptstyle\scriptscriptstyle{#1}}%
   \!\int}
\def\XXint#1#2#3{{\setbox0=\hbox{$#1{#2#3}{\int}$}
     \vcenter{\hbox{$#2#3$}}\kern-.5\wd0}}

\def\avgint{\Xint-}

\DeclareMathOperator{\Div}{div}
\DeclareMathOperator{\supp}{supp}
\DeclareMathOperator{\dist}{dist}

\newcommand{\op}{{\mathrm{op}}}

\DeclareMathOperator{\diag}{diag}
\DeclareMathOperator{\lip}{\mathrm{Lip}}
\newcommand{\R}{\mathbb{R}}

\newcommand{\loc}{\mathrm{loc}}

\newcommand{\F}{\mathcal F}

\newcommand{\vecb}{\mathbf b}
\newcommand{\vecc}{\mathbf c}
\newcommand{\vece}{\mathbf e}
\newcommand{\vecf}{\mathbf f}
\newcommand{\vecg}{\mathbf g}

\address{David Cruz-Uribe, OFS \\
Dept. of Mathematics \\
University of Alabama \\
 Tuscaloosa, AL 35487, USA}

\email{dcruzuribe@ua.edu}

\address{Feyza Elif Dal \\
Dept. of Mathematics \\
Y\i ld\i z Technical University \\
Esenler, Istanbul, 34220 Davutpasa, Turkey
}
\email{feyzadal@hotmail.com}

\address{Scott Rodney\\
Dept. of Mathematics, Physics and Geology \\ 
Cape Breton University \\
Sydney, NS B1Y3V3, CA} 

\email{scott\_rodney@cbu.ca}

\thanks{
The first author is partially supported by a Simons Foundation
  Travel Support for Mathematicians Grant and by NSF Grant DMS-2349550. The second author is supported by the TUBITAK 2211-E Domestic Direct Doctorate Scholarship Program and 2214-A International Research Fellowship Programme for PhD Students. The third  author  is partially supported by an NSERC development grant.  This project is supported by  TUBITAK, the Scientific and Technological Research Council of T\"urkiye through a 2501 Joint Research Program grant 223N112.}



\title[Degenerate Sobolev and Poincar\'e inequalities]
{Degenerate Sobolev and Poincar\'e inequalities \\ via extrapolation}

\author[Cruz-Uribe, Dal, Rodney] {David Cruz-Uribe, Feyza Elif Dal, Scott Rodney}
\begin{document}

\begin{abstract}
In this paper we prove matrix weighted Sobolev and Poincar\'e inequalities using techniques derived from the theory of Rubio de Francia extrapolation.  Given weights $w,\,v$ and a symmetric non-negative definite matrix valued function $Q$ defined on a connected open subset $\Omega$ of $\R^n$ that satisfies the lower ellipticity condition
\[  w(x)^p \leq |\sqrt{Q(x)}\xi|^p,\quad \xi\in \R^n, \]
we give Lebesgue integrability conditions on the weights $w,v$ that ensure there exists $\tau\geq 1$ so that Sobolev and Poincar\'e inequalities of the form
\[\bigg(\int_\Omega |u|^{\tau p} \,vdx\bigg)^{\frac{1}{\tau
      p}}
  \leq C(v,w)  \bigg(\int_\Omega |\sqrt{Q}\nabla u|^p\,dx\bigg)^{\frac{1}{ p}},\textrm{ and}\]
  \[\bigg(\int_\Omega |u-\langle u\rangle_{\Omega,v}|^{\tau p} \,v dx\bigg)^\frac{1}{\tau p} \\
    \leq C(v,w)\bigg(\int_\Omega|\sqrt{Q}\nabla u|^p \, dx\bigg)^{\frac{1}{p}}\]
    hold for smooth $u$.
    We explore these and related results in the context of several examples that include John domains, the Heisenberg group, and CR manifolds.
\end{abstract}
\maketitle

\section{Introduction}
\label{sec:introduction}

The purpose of this paper is to prove degenerate Sobolev and
Poincar\'e inequalities, e.g., inequalities of the form
\[ \bigg(\int_\Omega |u|^{\tau p} \,vdx\bigg)^{\frac{1}{\tau
      p}}
  \leq C(v,w)  \bigg(\int_\Omega |\sqrt{Q}\nabla
  u|^p\,dx\bigg)^{\frac{1}{ p}}, \]
where $v,\,w$ are weights and $Q$ is an $n\times n$, symmetric, positive semidefinite
matrix function that satisfies the degenerate lower ellipticity condition
\begin{equation*} 
   w(x)|\xi|^p\leq |\sqrt{Q(x)}\xi|^p, \quad \xi
   \in \R^n, \quad \text{a.e. } x\in \Omega.
   \end{equation*}
Our results are very general, and we can prove inequalities assuming
only integrability conditions on $v$ and $w^{-1}$.  Moreover, our
method of proof is novel (at least in this area):  we use techniques
developed in the study of Rubio de Francia extrapolation in harmonic
analysis (cf.~\cite{DCU-Martell-Perez,DCU}) that let us ``lift'' the
classical Sobolev and Poincar\'e inequalities and prove degenerate matrix weighted
inequalities.

The motivation for the study of degenerate Sobolev and Poincar\'e
inequalities comes from the study of degenerate elliptic partial
differential equations.  For context we give some of the history.
Fabes, Kenig, and Serapioni~\cite{Fabes-Kenig-Serapioni} studied the
Dirichlet problem for the operator $Lu=v^{-1}\Div(Q\nabla u)$ where
$Q$  satisfies
\begin{equation*}
    \lambda v(x)|\xi|^2\leq \langle Q(x)\xi,\xi \rangle\leq \Lambda
    v(x)|\xi|^2,
    \quad \xi
  \in \R^n, \quad \text{a.e. } x\in \Omega,
\end{equation*}
with  $\lambda$ and $\Lambda$ constants, and $v$ an element of the
Muckenhoupt class $A_2$, i.e,
\[ \sup_B \avgint_B v\,dx \avgint_B v^{-1}\,dx < \infty. \]
For $1<p<\infty$, using the properties of $A_p$ weights and the theory of weighted norm
inequalities, they proved the existence of a local weighted Sobolev
inequality:  given any
ball $B=B(x_0,r) \subset \R^n$,  and a weight $v\in A_p$, there exists $\delta>0$ such that
\begin{equation*}
  \bigg(\frac{1}{v(B)}\int_B |u|^{\sigma p} \,v dx\bigg)^{\frac{1}{\sigma p}}
  \leq Cr\bigg(\frac{1}{v(B)}\int_B |\nabla u|^p \,v dx\bigg)^{\frac{1}{p}},
\end{equation*}
where $1<\sigma<\frac{n}{n-1}+\delta$.  (Note that
$\sigma<\frac{n}{n-p}$, the gain in the classical $L^p$ Sobolev inequality.)
They also proved a local weighted Poincar\'e inequality:  
\begin{equation*}
    \bigg(\frac{1}{v(B)}\int_B |u-\langle u\rangle_{B,v}|^{\sigma p}
    \,v dx\bigg)^{\frac{1}{\sigma p}}
    \leq Cr\bigg(\frac{1}{v(B)}\int_B|\nabla u|^p \,v dx\bigg)^{\frac{1}{p}}
  \end{equation*}
where $\langle u\rangle_{B,v}=\frac{1}{v(B)}\int_B u\,v dx$.  Using
these inequalities (with $p=2$) they extended the regularity theory for
uniformly elliptic operators to these degenerate elliptic operators.

\medskip

Chanillo and Wheeden~\cite{Chanillo-Wheeden,Chanillo-Wheeden2}  studied the same
operator, but assumed that the matrix $Q$ satisfied the two-weight
ellipticity condition
\begin{equation} \label{eqn:degen-ellipticity-intro}
w(x)|\xi|^2 \leq |\sqrt{Q(x)}\xi|^2 \leq v(x)|\xi|^2. 
\end{equation}
For this, they needed two-weight
local Sobolev and Poincar\'e inequalities.
Given $1<p<\infty$, a pair of weights $v,\,w$ is said to be
$p$-admissible if $w\in A_p$, $v$ is doubling, that is, for $x_0\in
\R^n$ and $r>0$, 
$ v(B(x_0,2r))\leq Cv(B(x_0,r))$, and there exists $\sigma>1$, such
that 
for $0<r<s<\infty$ and $x\in \R^n$,
\begin{equation*}
  \frac{s}{r}\bigg(\frac{v(B(x,s))}{v(B(x,r))}\bigg)^{\frac{1}{\sigma p}}
  \leq C\bigg(\frac{w(B(x,s))}{w(B(x,r))}\bigg)^{\frac{1}{p}}.
\end{equation*}
With these assumptions they proved the following inequalities:  for $r>0$ and $x_0\in \R^n$, 
\begin{equation*}
    \bigg(\frac{1}{v(B)}\int_B |u|^{\sigma p}
      \,vdx\bigg)^{\frac{1}{\sigma p}}
      \leq Cr\bigg(\frac{1}{w(B)}\int_B |\nabla u|^p\,wdx\bigg)^{\frac{1}{p}},
    \end{equation*}
    and
    \begin{equation*}
    \bigg(\frac{1}{v(B)}\int_B |u-\langle u\rangle_{B,v}|^{\sigma p}
    \,v dx\bigg)^{\frac{1}{\sigma p}}
    \leq Cr\bigg(\frac{1}{w(B)}\int_B|\nabla u|^p \,w dx\bigg)^{\frac{1}{p}}.
  \end{equation*}
Their results were further generalized by Franchi, Lu, and
Wheeden~\cite{Franchi-Lu-Wheeden}.  

  \medskip
  
Korobenko, Rios, Sawyer, and Shen~\cite{Korobenko-Rios-Sawyer-Shen2}
studied elliptic differential equations in the infinitely degenerate
regime:  that is, they assumed $Q$
satisfied~\eqref{eqn:degen-ellipticity-intro} with $v=1$ and $w$ is a
continuous function
that decays to $0$ more quickly than any polynomial.  In this setting
they proved the local $(1,1)$ Sobolev and Poincar\'e inequalities
\begin{gather*}
\int_{B(0,r)}|u|\,dx\leq
   Cr\int_{B(0,2r)}|\sqrt{Q}\nabla u|\,dx, \\
   \int_{B(0,r)}|u-\langle u \rangle_B|\,dx\leq
   Cr\int_{B(0,2r)}|\sqrt{Q}\nabla u|\,dx;
 \end{gather*}
 here the balls are defined with respect to the Carnot-Carath\'eodory metric. 
We note that to prove these results they imposed a number of very
technical structural conditions on $w$ and $Q$; we refer the reader to
their paper for details.

In the same paper they also considered Sobolev inequalities where the
``gain'' on the left-hand side is in an Orlicz space, rather than the
usual Lebesgue space.  The reason for this is that in the infinitely
degenerate regime, balls in the associated Carnot-Caratheodory metric are
non-doubling, and Korobenko, Maldonado, and
Rios~\cite{Korobenko-Maldonado-Rios} proved that the existence of a
Sobolev inequality with gain in the scale of Lebesgue spaces is
(essentially) equivalent to balls in the Carnot-Caratheodory metric being
doubling.  
As a replacement, they proved an Orlicz-Sobolev inequality of the form
\begin{equation*}
    \|u\|_{L^\Phi(\Omega)}\leq C\|\sqrt{Q}\nabla u\|_{L^1(\Omega)},
  \end{equation*}
  where $\Phi(t) \approx t\log(e+t)^\sigma$, $\sigma>1$.   Again, this
  inequality required a number of very specific assumptions on $w$,
  $Q$, and $\sigma$ and we refer the reader to their paper for
  details. 

% In \cite{Korobenko-Rios-Sawyer-Shen}, Korobenko, Rios, Sawyer and Shen studied on boundedness of solutions to infinitely degenerate elliptic equations. Let $\mathcal{Q}(x,u)$ be a matrix function, $Q$ be degenerate elliptic matrix, $\phi_0:\Omega\rightarrow\R$ and $ \vec\phi_1:\Omega\rightarrow\R^n$ be locally integrable data functions. For any $u\in \lip_0(\Theta)$, quasilinear infinitely degenerate elliptic equations of the form is
% \begin{equation}\label{quasi eqn}
%     -\Div \mathcal{Q}(x,u)=\phi_0-\Div_Q\vec\phi_1
% \end{equation}
% and it is locally bounded.
% They defined an Orlicz-Sobolev bump inequality if $\Phi$ is a Young function, $\R^n$ has certain non-doubling metric and $\Omega$ be a bounded domain in $\R^n$ such that  
% \begin{equation*} \Phi^{(-1)}\bigg(\int_\Theta\Phi(u)d\mu\bigg)\leq u(\rho)\|\sqrt{Q}\nabla u\|_{L^1(\Theta)},
% \end{equation*}
% where the function $u(r)$ is continuous, nondecreasing and satisfies $u(0)=0, u(\rho)\ge \rho$ for all $0\leq\rho\leq r$ then there exists a weak solution to equation \eqref{quasi eqn}. They proved if a Young function $\Phi$ is $K$-submultiplicative, then 
% \begin{equation*}
%     \Phi^{(-1)}\bigg(\int_\Theta\Phi(u)\,d\mu\bigg)\leq K \|u\|_{L^\Phi(\Theta)}
% \end{equation*}
% for some constant $K\ge1$.

  \medskip
  
  A different approach to studying degenerate elliptic equations was
  first proposed by Sawyer and
  Wheeden~\cite{Sawyer-Wheeden2,Sawyer-Wheeden} (motivated by earlier
  work of Guti\'errez and Lanconelli~\cite{MR2015404} and others), and
  then extensively developed by the third author and his
  collaborators~\cite{SR,Monticelli-SR,MR3388872}.  (This approach was
  also used in the first half of~\cite{Korobenko-Rios-Sawyer-Shen2}.)
  Essentially, their approach was to {\em assume} that the matrix $Q$
  satisfied
\[ w(x)|\xi|^2 \leq \langle Q(x)\xi,\xi\rangle \leq 1, \]
and that the global  Sobolev inequality
\[ \bigg(\int_\Omega |u|^{2\sigma } \,dx\bigg)^{\frac{1}{2\sigma}}
  \leq C  \bigg(\int_\Omega |\sqrt{Q}\nabla
  u|^2\,dx\bigg)^{\frac{1}{ 2}}, \]
held for some $\sigma>1$.  With these assumptions they studied the existence and
regularity of solutions of the corresponding degenerate elliptic
equations.  They also considered equations with lower order terms, and, with the assumption of $L^p$ inequalities, 
the corresponding degenerate $p$-Laplacian $L_pu=
-v^{-1}\Div(|\sqrt{Q}\nabla u|^{p-2}Q\nabla u)$.

The first and third authors and their collaborators have extended this
approach further.  The three authors of this paper, with \c{C}etin and
Zeren~\cite{SC-DCU-FED-SR-Zeren}, considered the existence and
uniqueness of solutions of a very general second order elliptic
equation with lower order terms.  Key to their results are three assumptions.  First, they assumed the matrix $Q$ satisfies the ellipticity condition~\eqref{eqn:degen-ellipticity-intro}.
Second, they assumed the existence of a
global  Sobolev inequality
\begin{equation} \label{eqn:gen-sobolev-intro}
    \bigg(\int_\Omega |u|^{2\sigma }\,vdx\bigg)^{\frac{1}{2\sigma }}
    \leq C\bigg(\int_\Omega |\sqrt{Q}\nabla u|^2\,dx\bigg)^{\frac{1}{2}}
  \end{equation}
for some $\sigma\geq 1$, or the existence of a global Sobolev inequality with gain in the scale of Orlicz spaces,
\begin{equation} \label{eqn:gen-sobolev-orlicz-intro}
\|u\|_{L^\Phi(v,\Omega)} \leq C\bigg(\int_\Omega |\sqrt{Q}\nabla u|^2\,dx\bigg)^{\frac{1}{2}}, 
\end{equation}
where $\Phi$ is a Young function as $\Phi(t)=t^2\log(e+t)^\tau$, $\tau>1$.  Third, they 
assumed a weak Poincar\'e inequality:  given any compact set $K\subset \Omega$ and any $\epsilon>0$,  there exists $\delta>0$ so that if $0<r<\delta$ and $x\in K$, then for every ball $B=B(x,r)$, 
\begin{equation} \label{eqn:weak-poincare-intro}
\int_B |u-\langle u \rangle_{B,v}|^2\, vdx 
\leq \epsilon \int_B |\sqrt{Q}\nabla u|^2\,dx.
\end{equation}
Such an inequality holds if, for example, there exists $\nu>0$,
\begin{equation} \label{eqn:alt-poincare-intro}
\int_B |u-\langle u \rangle_{B,v}|^2\, vdx 
\leq Cr^\nu \int_B |\sqrt{Q}\nabla u|^2\,dx.
\end{equation}
With these assumptions, they showed that if the coefficients of the first order terms satisfied integrability conditions related to the gain in the Sobolev inequality, unique solutions existed.  

Finally, the first and third authors and Macdonald ~\cite{DCU_Macdonald_SR}
considered the boundedness and exponential
integrability of the solutions of
\begin{equation*}
    -v^{-1}\Div(|\sqrt{Q}\nabla u|^{p-2}Q\nabla u)=f|f|^{p-2}.
  \end{equation*}
  where $Q$ satisfied the ellipticity condition~\eqref{eqn:degen-ellipticity-intro}.
  They showed that if they assumed a degenerate Sobolev inequality of
  the form~\eqref{eqn:gen-sobolev-intro}  or an inequality of the form~\eqref{eqn:gen-sobolev-orlicz-intro} (in both cases again with the power 2 replaced by $p$), then solutions
  were bounded, with the integrability of $f$ depending on $p$ and
  $\sigma$.   If they assumed a Sobolev inequality without gain (i.e., $\sigma=1$), they showed solutions were exponentially integrable.

 \medskip

Given this history, we now turn to  the goal of this paper: to prove
degenerate Sobolev and Poincar\'e inequalities of the kinds assumed in
the approach developed by Sawyer and Wheeden.  We will do so assuming
only integrability conditions on $v$ and $w^{-1}$; thus, our inequalities 
 require much weaker assumptions that those of Fabes, Kenig, and
Serapioni~\cite{Fabes-Kenig-Serapioni} and Chanillo and
Wheeden~\cite{Chanillo-Wheeden2,Chanillo-Wheeden}.

\begin{remark}
We note that our results cannot be applied to prove Sobolev and Poincar\'e inequalities like those proved by Korobenko, {\em et al.}.   Because they are working in the infinitely degenerate regime, $w^{-1}$ is not in any $L^p$ space, $p>0$.  This obstruction seems intrinsic to our proofs; it is an open question whether our techniques can be further extended to prove inequalities similar to theirs.    
\end{remark}

The remainder of this paper is organized as follows.  In
Section~\ref{section:State. Main Results} we give an abstract
extrapolation result.  Following the approach to Rubio de Francia
extrapolation developed in~\cite{DCU-Martell-Perez}, it is stated for
an abstract family of pairs of functions referred to as extrapolation pairs.  It is
this machinery that will let us prove a variety of Sobolev and
Poincar\'e inequalities as immediate corollaries.  We also prove two
variants of this result; the first, Theorem~\ref{Main Thm 3}, lets us
extrapolate from a ``weak'' Sobolev inequality, and the second,
Theorem~\ref{Orlicz case}, lets us prove Sobolev and Poincar\'e
inequalities in the scale of Orlicz spaces.  For simplicity we state and prove these results in Euclidean space with Lebesgue measure, but it will be clear that they hold in any measure space.

In Section~\ref{section:prelim} we define the weighted Lebesgue
spaces, Orlicz spaces, and Sobolev spaces that appear in our results.
We also gather together some basic results about Orlicz spaces that we
will need in the proof of Theorem~\ref{Orlicz case}.   For the
degenerate Sobolev spaces, we follow the development and notation used
in~\cite{SC-DCU-FED-SR-Zeren}, and we refer the reader there for
further information.  

In Section~\ref{section:Applications} we first use our extrapolation
theorem to give a degenerate global Sobolev inequality,
Theorem~\ref{global sob.}.  We also give a special case of
Theorem~\ref{global sob.}, Theorem~\ref{classical Sob.}.  This result
was first proved in~\cite{DCU-FED-SR-Zeren}; we include it here to
show how it fits in our broader framework.  We then give an
Orlicz-Sobolev inequality, Theorem~\ref{Orlicz case cor}.  Next, we
give a local Poincar\'e inequality on balls, Theorems~\ref{local
  Poin.}, and a global Poincar\'e inequality on $s$-John domains,
Theorem~\ref{s-John}.  We also show how our results can be extended to
other geometries, and we give degenerate Folland-Stein type
inequalities on the Heisenberg group, Theorem~\ref{Heisenberg Group}, and for compact manifolds,
Theorems~\ref{riemannian} and~\ref{Heis. Gr. Cor. 2}.

We emphasize to the reader that the degenerate Sobolev and Poincar\'e
inequalities in our results are meant to be illustrative, and not
exhaustive.  They were chosen to demonstrate the flexibility of our
extrapolation theorems and we believe that these will find
applications to prove other kinds of inequalities.  In particular, the Folland-Stein type inequalities should be thought of as model results in other geometries.

In Section~\ref{section:pdes} we briefly consider the application of
our results to the study of partial differential equations by showing
how we can construct weights $v,\,w$ and matrices $Q$ so that the corresponding Sobolev and Poincar\'e inequalities satisfy the hypotheses of the existence results
in~\cite{SC-DCU-FED-SR-Zeren} and the boundedness results
in~\cite{DCU_Macdonald_SR} discussed above.
Finally, in Sections~\ref{section:proof-extrapol}
and~\ref{section:Proof of App.} we prove the results in
Sections~\ref{section:State. Main Results}
and~\ref{section:Applications}.

Throughout this paper we will use the following notation.
The constant $n\geq 2$ will denote the dimension of the underlying space
$\R^n$; $\Omega\subset \R^n$ will be an open, connected set.  The
matrix $Q$ will be an $n\times n$, measurable, symmetric, positive
semidefinite matrix.  We define $\sqrt{Q}$ in the standard way; see \cite[Section 2]{SC-DCU-FED-SR-Zeren}.   By a
weight we mean a non-negative, locally integrable function.
Constants will be denoted by $C$, $c$, etc. and may change at each
appearance.  They can depend on the underlying parameters, but not on
the weights $v,\,w$ or the matrix $Q$; we will explicitly track the dependence on these quantities.

\section{The extrapolation theorems}
\label{section:State. Main Results}

In this section, we state our extrapolation theorems.  The underlying
philosophy of each result is straightforward:  given an unweighted
inequality (e.g., the classical Sobolev inequality) show that it can
be lifted to a degenerate (that is, a weighted) inequality.  Our main tool in the proofs will
be H\"older's inequality, and the hypotheses are precisely those
needed to apply it.

So that our results can be applied to prove a wide variety of
inequalities in different contexts, we adopt the following convention from
the theory of Rubio de Francia extrapolation.  We define an abstract
family $\F$ to consist of pairs $(f,F)$ of measurable functions, with 
 $f:\Omega\rightarrow [0,\infty]$  a real-valued measurable function
 and $F:\Omega\rightarrow \R^n$  a vector-valued measurable function.
Given function spaces $X$ and $Y$, when we write an inequality of the
form
\begin{equation} \label{eqn:extrapol-conv}
     \|f\|_X\ \leq c_0 \|F\|_Y,\qquad (f,F) \in \mathcal{F},
   \end{equation}
 we mean that this inequality holds for all pairs $(f,F)\in
 \mathcal{F}$ such that $\|f\|_X<\infty$.  A similar convention holds for Theorem~\ref{Main Thm 3} below with an additional term on the right-hand side.

 \begin{remark}  \label{remark:two-step}
   As we will see in Section~\ref{section:Applications}, to apply our
 extrapolation results, we need to construct the family $\F$ so that
 \eqref{eqn:extrapol-conv} holds for both the ``base case'' (an unweighted Sobolev or Poincar\'e inequality)
  and for the desired degenerate  inequality.   The
   inequalities we are proving hold for functions that may not be in
   the family $\F$; thus, each proof will require an approximation
   argument.  This is intrinsic to this approach, and occurs in Rubio
   de Francia extrapolation as well. 
 \end{remark}
 
 \medskip

We now state our results.  In each extrapolation theorem, given a
Lebesgue space $L^{\frac{a}{b}}(\Omega)$, if $b=0$, then we interpret
this as $L^\infty(\Omega)$.  We adopt the analogous convention for
Orlicz space norms.

\begin{theorem}\label{extrapolation}
  Given $1\leq q< \infty$ and $\sigma \ge1$, suppose that the family
  of extrapolation pairs $\F$ satisfies the inequality
\begin{align}\label{Ineq with sigma q in gen.}
    \bigg(\int_\Omega |f|^{\sigma q} \, dx\bigg)^\frac{1}{\sigma q} \leq c_0\bigg(\int_\Omega \big |F \big|^q\, dx\bigg)^\frac{1}{q}, \qquad (f,F) \in \mathcal{F},   
   \end{align}
   where the constant $c_0>0$ may depend on $\F$ but is uniform for
   each pair $(f,F)\in \F$.  Let $1\leq p <\infty$ and $1\leq \tau
   \leq \sigma$ be such that
   \begin{equation}\label{prop for 4.1}
       \frac{\tau p}{\sigma} \leq q \leq p.
     \end{equation}
   Given weights $v,\,w$, if $v\in L^{\frac{\sigma q}{\sigma q-\tau p}}(\Omega)\, $, $
   \,w^{-1} \in L^{\frac{q}{p-q}}(\Omega)$, and the matrix $Q$
   satisfies the lower ellipticity condition
    \begin{equation} \label{ellptcy for gain}
w(x)|\xi|^p\leq |\sqrt{Q(x)}\xi|^p, \quad \xi \in \R^n, \quad
\text{a.e. } x \in \Omega,
\end{equation}
then
\begin{multline}\label{Ineq for gen.}
  \bigg(\int_\Omega |f|^{\tau p} \,v dx\bigg)^\frac{1}{\tau p} \\
  \leq c_0\|w^{-1}\|^{\frac{1}{p}}_{L^{\frac{q}{p-q}}(\Omega)}
  \|v\|^{\frac{1}{\tau p }}_{L^{\frac{\sigma q}{\sigma
        q-\tau p}}(\Omega)}
  \bigg(\int_\Omega|\sqrt{Q}F|^p \, dx\bigg)^{\frac{1}{p}},
  \qquad (f,F) \in \mathcal{F}.
\end{multline}
\end{theorem}

\begin{remark} \label{remark:measure-space}
  While  Theorem~\ref{extrapolation}, as well as Theorems~\ref{Main Thm 3}
  and~\ref{Orlicz case} below, are stated and proved in Euclidean space with
  Lebesgue measure, the same results are true on any measure space
  $(\Omega,\Sigma,\mu)$.  For Theorem~\ref{Orlicz case} we only need
  to make the minimal assumptions required for Orlicz spaces to be
  well-defined.  See Rao and Ren~\cite{Rao-Ren} for more information.
  We will use this fact to state and prove degenerate Folland-Stein inequalities on
  the Heisenberg group and CR manifolds.
\end{remark}

 We next consider extrapolation from a more complicated inequality which has  an additional term to the
  right-hand side.  

 \begin{theorem}\label{Main Thm 3}
   Given $1\leq q< \infty$ and $\sigma \ge1$, suppose that the family
  of extrapolation pairs $\F$ satisfies the inequality
   \begin{equation}\label{Ineq with sigma q for MT3}
    \bigg(\int_\Omega |f|^{\sigma q} \, dx\bigg)^\frac{1}{\sigma q} \leq c_0\bigg(\int_\Omega \big |F \big|^q\, dx\bigg)^\frac{1}{q}+c_1\bigg(\int_\Omega |f|^q \,dx \bigg)^{\frac{1}{q}}, \qquad (f,F) \in \mathcal{F},   
  \end{equation}
  where the constants $c_0,\,c_1 >0$ may depend on $\F$ but are
  uniform for each pair $(f,F)\in \F$.  Let $1\leq p <\infty$ and
  $1\leq \tau \leq \sigma$ be such that \eqref{prop for 4.1} holds.
 Given weights $v,\,w$,  if $v\in L^{\frac{\sigma q}{\sigma q-\tau p}}(\Omega)\,
  $and$ \,w^{-1} \in L^{\frac{q}{p-q}}(\Omega)$,  and the matrix $Q$
  satisfies the lower ellipticity condition~\eqref{ellptcy for gain},
   then
   \begin{multline}\label{Ineq for MT3}
    \bigg(\int_\Omega |f|^{\tau p} \,v dx\bigg)^\frac{1}{\tau p}
    \\
    \leq \|w^{-1}\|^{\frac{1}{p}}_{L^{\frac{q}{p-q}}(\Omega)}
    \|v\|^{\frac{1}{\tau p }}_{L^{\frac{\sigma q}{\sigma q-\tau p}}(\Omega)}
    \bigg[c_0\bigg(\int_\Omega|\sqrt{Q}F|^p \, dx\bigg)^{\frac{1}{p}}+c_1\bigg(\int_\Omega|f|^p \,w dx \bigg)^{\frac{1}{p}}\bigg], \, (f,F) \in \mathcal{F}.
  \end{multline}
 \end{theorem}

\begin{remark}
    In practice, we will often assume the stronger ellipticity condition
   \begin{equation}\label{elpt con.}
     w(x)|\xi|^p\leq |\sqrt{Q(x)}\xi|^p\leq v(x)|\xi|^p,
     \quad \xi\in \R^n, \quad \text{a.e. } x\in \Omega,
     \end{equation}
in which case $w(x)\leq v(x)$ a.e., so we can take the $L^p(v,\Omega)$ norm of $f$ on the right-hand side of~\eqref{Ineq for MT3}.
\end{remark}

 \begin{remark}
   While the powers in the second term on the right-hand side of
   inequalities~\eqref{Ineq with sigma q for MT3} and~\eqref{Ineq for
     MT3} are the same as in the first term, it will be clear from the
   proof that they can be different, possibly with a different
   hypothesis on $w$.  Details are left to the interested reader.
 \end{remark}

 In our final  result, we show that we can extrapolate
 into the scale of Orlicz spaces.  For brevity, we defer the precise
 definition of these spaces to Section~\ref{section:prelim} below.

\begin{theorem}\label{Orlicz case}
  Given $1\leq q< \infty$ and $\sigma >1$, suppose that the family
  of extrapolation pairs $\F$ satisfies inequality~\eqref{Ineq with
    sigma q in gen.}, where the constant $c_0 >0$ may depend on
  $\F$ but is uniform for each pair $(f,F)\in \F$.  Fix $p$,
  $q\leq p < \sigma q$, and $\tau>0$, and define
  $A(t)=t^p\log(e+t)^\tau$. Given weights $v,\,w$, if $v\in L^\Psi(\Omega)$,
   where
  $\Psi(t)=t^{\frac{\sigma q}{\sigma
      q-p}}\log(e+t)^{\frac{\tau\sigma q}{\sigma q-p}}$, $w^{-1}\in L^{\frac{q}{p-q}}(\Omega)$, and the
  matrix $Q$ satisfies the lower ellipticity condition
  \eqref{ellptcy for gain}, then
  \begin{multline}\label{Sob. Ineq Orlicz}
    \|f\|_{L^A(v,\Omega)} \\
    \leq 4c_0\|w^{-1}\|^{\frac{1}{p}}_{L^{\frac{q}{p-q}}(\Omega)}
    \|v\|^{\frac{1}{\sigma q}}_{L^\Psi(\Omega)}
    (1+\|v\|_{L^\Psi(\Omega)})^{\frac{1}{\sigma q}}
    \bigg(\int_\Omega|\sqrt{Q}F|^p \, dx\bigg)^{\frac{1}{p}},
    \qquad (f,F) \in \mathcal{F}.
  \end{multline}
\end{theorem}

\begin{remark}
Since Orlicz norms represent an endpoint case, the conditions in Theorem~\ref{Orlicz case} are required to be somewhat stronger than in Theorem~\ref{extrapolation}.  Thus, while in that result we can take $\tau p=\sigma q$, here the proof requires that $p<\sigma q$ for any value of $\tau>0$.  In particular, this requires that we assume that $\sigma>1$.
    \end{remark}

\begin{remark}
  The nonlinear form of the constant involving the $L^\Phi$ norm of
  $v$ is a consequence of the proof, and comes from estimating the norm of
  $v$ in $L^\Psi(v,\Omega)$.  Such an estimate is immediate in a
  Lebesgue space, but is more complicated in an Orlicz space.
\end{remark}

\begin{remark}
    We have proved Theorem~\ref{Orlicz case} for Young functions of the form $\Phi(t)=t^p\log(e+t)^q$ since these generate the Orlicz norms that have appeared frequently in the study of degenerate PDEs.  However, the proof can be readily modified to work with other families of Young functions; for an example of such families, see~\cite{10.48550/arxiv.2210.12441}.  We leave the details to the interested reader.
\end{remark}

\begin{remark} 
  In the proof of Theorem~\ref{Orlicz case} we could replace the Orlicz
  spaces with other Banach function spaces.  For instance, we could
  use this to prove degenerate Sobolev and Poincar\'e inequalities in
  variable Lebesgue spaces.  Such results have been considered
  recently in several papers:
  see~\cite{DCU-Penrod_SR,MR5020484,MR4641614,Almeida:2026vi}.  Details of
  such extensions are left to the interested reader.
\end{remark}

\section{Preliminaries on function spaces}
\label{section:prelim}

In this section we give some definitions and basic results about the
function spaces that appear in our results.   For the results on
weighted Lebesgue and Sobolev spaces, see~\cite{SC-DCU-FED-SR-Zeren};
for Orlicz spaces, see~\cite{Rao-Ren}.

\subsection*{Lebesgue spaces and Lipschitz spaces}
Given a weight $v$ and $1\leq p<\infty$, $L^p(v,\Omega)$ is the
Banach space of measurable functions defined on $\Omega$ with norm
\begin{equation*}
    \|f\|_{L^p(v,\Omega)}=\bigg(\int_\Omega |f|^p  \,v dx\bigg)^{1/p}<\infty.
  \end{equation*}
  %
%$L^p(v,\Omega)$ is a Banach space with respect to this norm. 
Similarly, given a matrix $Q$, the space $QL^p(\Omega)$ is the
collection of measurable functions $\vecf :\Omega\rightarrow \R^n$
such that
\begin{equation*}
\|\vecf\|_{QL^p(\Omega)}=\bigg(\int_\Omega |\sqrt{Q}\vecf|^p  \,dx\bigg)^{1/p}<\infty.
\end{equation*}
$QL^p(\Omega)$ is also a Banach space under the equivalence $\vecf=\vecg$ if
and only if $\|\vecf-\vecg\|_{QL^p(\Omega)}=0$.

The space $\lip_\loc(\Omega)$ consists of functions that are Lipschitz
on compact subsets of $\Omega$. The space $\lip_0(\Omega)$ is the
collection of Lipschitz functions which are supported on compact
subsets of $\Omega$. The space $\lip(\overline{\Omega})$ is
the collection of functions that are Lipschitz on $\overline{\Omega}$.
Note that
$\lip_0(\Omega) \subset \lip(\overline{\Omega})\subset\lip_\loc(\Omega)$.

\subsection*{Degenerate Sobolev spaces}
Given $1\leq p<\infty$, a matrix $Q$, and a weight $v$, define
$Q\lip_\loc(v,\Omega)$ to be the collection of all functions
$\varphi \in \lip_\loc(\Omega)$ such that
     \begin{equation*}
    \|\varphi\|_{QH^{1,p}(v,\Omega)}=\|\varphi\|_{L^p(v,\Omega)}+\|\nabla \varphi\|_{QL^p(\Omega)}<\infty.
  \end{equation*}
Define the space $QH^{1,p}(v,\Omega)$ to be  the abstract closure of
$Q\lip_\loc(v,\Omega)$ with respect to this norm. Formally,
$QH^{1,p}(v,\Omega)$  is the
  collection of equivalence classes of Cauchy sequences of
  $Q\lip_\loc$ functions.  However, since $L^p(v,\Omega)\times
  QL^p(\Omega)$ is a Banach space, to each such equivalence class we
  can associate a unique pair $(u,\vecg)$.  By an abuse of notation,
  we write $\nabla u = \vecg$, and refer to it as the degenerate weak
  gradient of $u$.  It is important to note that $\nabla u$ may not be
  the weak gradient of $u$ in the sense of distributional
  derivatives.  See~\cite{SC-DCU-FED-SR-Zeren} for more details.  
  
  We define two subspaces of $QH^{1,p}(v,\Omega)$.  Let 
  $Q\lip_0(v,\Omega)=\lip_0(\Omega)\cap Q\lip_\loc(v,\Omega)$ and
  define $QH_0^{1,p}(v,\Omega)$ to be the closure of $Q\lip_0(\Omega)$
  in $QH^{1,p}(v,\Omega)$.   Similarly, let
  $Q\lip(v,\overline{\Omega})= \lip(\overline{\Omega})\cap
  Q\lip_\loc(v,\Omega)$, and let $QH^{1,p}(v,\overline{\Omega})$
  be the closure of $Q\lip(v,\overline{\Omega})$ in
  $QH^{1,p}(v,\Omega)$.

  \begin{remark}
    If $v\in L^1(\Omega)$ and is the largest eigenvalue of $Q$, then
    $|Q|_\op \leq v$, and so $Q\lip_0(v,\Omega)= \lip_0(\Omega)$ and 
    $Q\lip(v,\overline{\Omega})= \lip(\overline{\Omega})$.
  \end{remark}

\subsection*{Orlicz spaces}
A Young function $\Psi:[0,\infty)\rightarrow[0,\infty)$ is a convex, continuous, and strictly increasing function such that $\Psi(0)=0$ and $\Psi(t)/t\rightarrow\infty$ as $t\rightarrow \infty$.  We will denote Young functions by the Greek letters $\Gamma, \Phi$ and $\Psi$, and also using upper case letters $A, B$, etc.

Given a Young function $\Psi$ and a weight $v$,  define the weighted Orlicz space $L^\Psi(v,\Omega)$ to  be the collection of  measurable functions $f$  defined on $\Omega$ such that 
\begin{equation}\label{Orlicz norm}
    \|f\|_{L^\Psi(v,\Omega)}=\inf\bigg\{\lambda>0:\int_\Omega\Psi\bigg(\frac{|f|}{\lambda}\, \bigg)\, vdx \leq1\bigg\}<\infty.
\end{equation}
With this norm $L^\Psi(v,\Omega)$ is a Banach space.  

If $\Psi(t)=t^p, 1\leq p<\infty$, then $L^\Psi(v,\Omega)=L^p(v,\Omega)$, with equality of norms. More generally, given two  Young functions $\Psi$ and $\Gamma$, if there exists a constant $C>0$ such that  for all $t>0$, $\Psi(t)\leq C\Gamma(t)$, then $\|f\|_{L^\Psi(v,\Omega)}\lesssim\|f\|_{L^\Gamma(v,\Omega)}$. If we also have that $c\Gamma(t) \leq \Psi(t)$ for some $c>0$, we write  $\Psi(t)\approx\Gamma(t)$.  In this case $\|f\|_{L^\Psi(v,\Omega)}\approx\|f\|_{L^\Gamma(v,\Omega)}$. The implicit constants depend $\Psi$ and $\Gamma$. 

The Orlicz norm satisfies a rescaling property: given  a young function $\Phi$, suppose that for some $s>0$, $\Gamma(t)=\Psi(t^s)$ is also a Young function.  Then it follows immediately from the definition that
\begin{equation}\label{rescaling}
    \||f|^s\|_{L^\Psi(v,\Omega)}=\|f\|^s_{L^\Gamma(v,\Omega)}.
\end{equation}

There is an equivalent expression for the Orlicz norm. Given a Young function $\Psi$ we define the Amemiya norm,
\begin{equation*}\label{Orl. Norm2}
    \|f\|^*_{L^\Psi(v, \Omega)}
    =\inf\bigg\{\lambda+\lambda\int_\Omega \Psi\bigg(\frac{|f|}{\lambda}\, \bigg)\,vdx:\lambda>0\bigg\}.
\end{equation*}
Then, for $f\in L^\Psi(v,\Omega)$,
\begin{equation}\label{Orl Norm Rel.}
\|f\|_{L^\Psi(v,\Omega)}\leq\|f\|_{L^\Psi(v,\Omega)}^*\leq 2\|f\|_{L^\Psi(v,\Omega)}.    
\end{equation}
See~\cite[Section 3.3, Theorem 13]{Rao-Ren}.

Unlike the $L^p$ norms, the Orlicz norm cannot be expressed simply in terms of an integral.  However, we do have the following inequality: given a Young function $\Psi$, if $0<\|f\|_{L^\Psi(\Omega)}\leq 1$, then
\begin{equation} \label{eqn:modular}
\int_\Omega \Psi(|f|)\,dx \leq 1. 
\end{equation}
See~\cite[Section~3.4, Proposition~6]{Rao-Ren}.

Given a Young function $\Psi$, define its complementary function $\Bar{\Psi}$ by
\begin{equation*}
    \Bar{\Psi}(t)=\sup_{r>0} \{rt-\Psi(r)\};
\end{equation*}
$\Psi$ and $\Bar{\Psi}$  satisfy the  pointwise relationship
$t\leq \Psi^{-1}(t)\Bar{\Psi}^{-1}(t)\leq 2t$.
In the proof of Theorem~\ref{Orlicz case}, we need the complementary function of a specific Young function. Using this formula, it is straightforward to show that if $\Psi(t)=t^p\log(e+t)^q, 1<p<\infty$ and $0<q<\infty$, then
\begin{equation} \label{eqn:conjugate-young}
    \Bar{\Psi}(t)\approx \frac{t^{p'}}{\log(e+t)^{q(p'-1)}}.
\end{equation}

H\"older's inequality is true in the scale of Orlicz spaces:  given a Young function $\Phi$, 
\begin{equation}\label{Hölders Ineq for Orlicz}
    \int_\Omega|fg|v \,dx\leq 2\|f\|_{L^\Phi(v,\Omega)}\|g\|_{L^{\Bar{\Phi}}(v,\Omega)}.
\end{equation}
We also have the following duality result:  given a Young function $\Phi$, define
\[ \|f\|_{L^\Phi(v,\Omega)}' = \sup\bigg\{ \int_\Omega fg\,vdx : 
\|g\|_{L^{\bar{\Phi}}(v,\Omega)}\leq 1\ \bigg\}.  \]
Then 
\begin{equation} \label{eqn:orlicz-dual}
\|f\|_{L^\Phi(v,\Omega)} \leq \|f\|_{L^\Phi(v,\Omega)}' \leq 2\|f\|_{L^\Phi(v,\Omega)}. 
\end{equation}

\section{Sobolev and Poincar\'e inequalities via extrapolation}
\label{section:Applications}

In this section we use Theorems~\ref{extrapolation}, \ref{Main Thm 3}, and~\ref{Orlicz case} to prove degenerate Sobolev inequalities and degenerate Poincar\'e inequalities.  We will also  use them to prove degenerate Folland-Stein type inequalities on the Heisenberg group and on CR manifolds.  

\subsection*{Sobolev inequalities}
We first consider degenerate Sobolev inequalities.  The starting point for our first result (the ``base case" used in Theorem~\ref{extrapolation}) is the classical Sobolev inequality (see~\cite[Theorem~7.10]{Gilbarg-Trudinger}):
\begin{equation} \label{eqn:classical-sobolev}
    \bigg(\int_\Omega |\varphi|^{\sigma q}\,dx\bigg)^{\frac{1}{\sigma q}}
    \leq S(q,\sigma)\bigg(\int_\Omega |\nabla \varphi|^q\,dx\bigg)^{\frac{1}{q}},
\end{equation}
where $\varphi \in \lip_0(\Omega)$, and for $1\leq q<n$, $\sigma = \frac{n}{n-q}$.  Moreover, if we assume that $\Omega$ is bounded, then by H\"older's inequality, for $1\leq q<n$, we can take $1\leq \sigma \leq \frac{n}{n-q}$, and for $q\geq n$, we can take $1\leq \sigma < \infty$.  

Our first result is a degenerate global Sobolev inequality; the hypotheses are somewhat complicated as we have to distinguish between bounded and unbounded domains in order to use inequality~\eqref{eqn:classical-sobolev}.  For applications of this result, see Section~\ref{section:pdes} below.

\begin{theorem}\label{global sob.}
    Given a domain $\Omega$, we consider three cases:
    \begin{enumerate}
        \item If $\Omega$ is unbounded, fix any $p$, $1 \leq p<n$.  Choose $q$, $1\leq q\leq p<n$, let $\sigma=\frac{n}{n-q}$, and choose $\tau$, so that $1\leq \tau\leq\frac{\sigma q}{p}$.

        \item If $\Omega$ is bounded,  fix any $p$, $1 \leq p<n$.  Choose $q$, $1\leq q\leq p<n$, choose $\sigma$, $1\leq \sigma \leq \frac{n}{n-q}$, and choose $\tau$, so that $1\leq \tau\leq\frac{\sigma q}{p}$.
        
        \item If $\Omega$ is bounded,  fix any $p$, $n\leq p<\infty$.  Choose $\tau$, $\tau\ge1$, choose $q$, $n\leq q\leq p$ and choose $\sigma$, $\sigma\ge1$ so that $1\leq \tau\leq \frac{\sigma q}{p}$.
    \end{enumerate}
    Given weights $v$ and $w$, if $v\in L^{\frac{\sigma q}{\sigma q-\tau p}}(\Omega)$ and $w^{-1} \in L^{\frac{q}{p-q}}(\Omega)$, and the matrix $Q$ satisfies the lower ellipticity condition~\eqref{ellptcy for gain}, 
    then for any $u\in QH_0^{1,p}(v,\Omega)$,
\begin{equation}\label{Ineq for u}
    \bigg(\int_\Omega |u|^{\tau p} \,v dx\bigg)^\frac{1}{\tau p}
    \leq S(q,\sigma)\|w^{-1}\|^{\frac{1}{p}}_{L^{\frac{q}{p-q}}(\Omega)}
    \|v\|^{\frac{1}{\tau p }}_{L^{\frac{\sigma q}{\sigma q-\tau p}}(\Omega)}
    \bigg(\int_\Omega|\sqrt{Q}\nabla u|^p \, dx\bigg)^{\frac{1}{p}},
\end{equation} 
where $S(q,\sigma)$ is the constant in~\eqref{eqn:classical-sobolev}.
\end{theorem}

\begin{remark}
Note that the range of possible values for $q$ is restricted by the choice  of $p$ and the value of $\sigma$.  For instance, if $\sigma=\frac{n}{n-q}$, then to insure that $\frac{q\sigma}{p}\geq 1$ (so that a value of $\tau\geq 1$ exists), we must have that $q\geq \frac{pn}{n+p}$.  This lower bound is greater than $1$ if $p>\frac{n}{n-1}$. For smaller values of $\sigma$, $q$ must be chosen even closer to $p$ in value.  The same phenomenon occurs in the other Sobolev and Poincar\'e inequalities we state below.
\end{remark}

The following result was proved in~\cite[Theorem 1.1]{DCU-FED-SR-Zeren} and was the motivation for this paper.  It is a special case of Theorem~\ref{global sob.}; we include it here to show how it fits into our general framework.

\begin{theorem}\label{classical Sob.}
Given a domain $\Omega$ and  $1<p<\infty$, let $v\in L^n(\Omega)$ be a weight and let $Q$ be a matrix function that satisfies the lower ellipticity condition 
     \begin{equation}\label{elp.cond.}
    \frac{1}{v(x)}|\xi|^{p'}\leq |\sqrt{Q(x)}\xi|^{p'}, \xi \in \R^n \text{ and a.e. } x\in \Omega.
\end{equation}
   Then for any $u\in QH_0^{1,p}(v,\Omega)$,
   \begin{equation}\label{deg. Sob.}
        \bigg(\int_\Omega |u|^{p} \,v dx\bigg)^\frac{1}{p}\leq S(q,\sigma)\|v\|_{L^{n}(\Omega)}\bigg(\int_\Omega|\sqrt{Q}\nabla u|^p \, dx\bigg)^{\frac{1}{p}},
   \end{equation}
where $S(q,\sigma)$ is the constant in~\eqref{eqn:classical-sobolev}, $q=\frac{np}{n+p-1}$, and $\sigma=\frac{n}{n-q}$. 
    \end{theorem}

Finally, we give sufficient conditions for a degenerate Sobolev inequality with gain in the scale of Orlicz spaces to hold.  For simplicity we only consider the case of bounded domains.  For unbounded domains a similar result holds, with hypotheses analagous to those in Theorem~\ref{global sob.}.  Details are left to the interested reader.

\begin{theorem}\label{Orlicz case cor}
Given a bounded domain $\Omega$ and $1\leq p<n$, take any $q$, $1\leq q\leq p$, take any $\sigma$, $1< \sigma \leq \frac{n}{n-q}$, so that $p< \sigma q$.  Fix  any  $\tau >0$ and define  $A(t)=t^p\log(e+t)^\tau$ and $\Psi(t)=t^\frac{\sigma q}{\sigma q-p}\log(e+t)^{\frac{\tau \sigma q}{\sigma q-p}}$.  Given weights $v$ and $w$, if $v\in L^\Psi(\Omega)$ and $w^{-1}\in L^{\frac{q}{p-q}}(\Omega)$, and the matrix $Q$ satisfies the lower ellipticity condition \eqref{ellptcy for gain}, then for any $u\in QH^{1,p}_0(v,\Omega)$,
    \begin{equation}\label{Orlicz ineq.}
        \|u\|_{L^A(v,\Omega)}\leq 4S(q,\sigma)\|w^{-1}\|^{\frac{1}{p}}_{L^{\frac{q}{p-q}}(\Omega)}\|v\|^{\frac{1}{\sigma q}}_{L^\Psi(\Omega)}(1+\|v\|_{L^\Psi(\Omega)})^{\frac{1}{\sigma q}}\bigg(\int_\Omega|\sqrt{Q}\nabla u|^p \, dx \bigg)^{\frac{1}{p}}.
    \end{equation}
\end{theorem}

\medskip

\subsection*{Poincar\'e inequalities}
We now consider degenerate Poincar\'e inequalities.   Suppose first that the domain $\Omega$ is a ball $B=B(x_0,r)$, $x_0\in \R^n$, $r>0$.  In this case we have that if $1\leq q<n$ and $1\leq \sigma \leq \frac{n}{n-q}$, or if $q\geq n$ and $1\leq \sigma <\infty$, then for all $\varphi \in \lip(\overline{B})$,
\begin{equation} \label{eqn:poincare-ball} 
\bigg(\frac{1}{|B|} \int_B |\varphi-\langle \varphi \rangle_B|^{\sigma q} \,dx \bigg)^{\frac{1}{\sigma q}}
\leq  P(q,\sigma) r\bigg(\frac{1}{|B|} \int_B |\nabla \varphi|^q\,dx \bigg)^{\frac{1}{q}}, 
\end{equation}
where $\langle \varphi \rangle_B=\frac{1}{|B|}\int_B \varphi\,dx$.  When $\sigma=1$ this result is well known:  see, for instance,~\cite[Section~7.8, (7.45)]{Gilbarg-Trudinger}.  For $\sigma=\frac{n}{n-q}$, this inequality is sometimes referred to as a Sobolev-Poincar\'e inequality.  It follows from the case when $\sigma=1$ by the Rellich-Kondrachov theorem:  see~\cite[Section~4.1]{MR1014685}.  As with the Sobolev inequality, the cases when $\sigma < \frac{n}{n-q}$ or $q\geq n$ and $\sigma\geq 1$ follow by H\"older's inequality.  Note that in this case the constant $P(q,\sigma)$ also depends on $|\Omega|$.

To more readily apply extrapolation, we rewrite inequality~\eqref{eqn:poincare-ball} as follows:
\begin{equation} \label{eqn:poincare-ball-alt} 
\bigg( \int_B |\varphi-\langle \varphi \rangle_B|^{\sigma q} \,dx \bigg)^{\frac{1}{\sigma q}}
\leq P(q,\sigma) r |B|^{\frac{1}{\sigma q}-\frac{1}{q}}\bigg( \int_B |\nabla \varphi|^q\,dx \bigg)^{\frac{1}{q}}.
\end{equation}

\begin{theorem}\label{local Poin.} 
Given a ball $B=B(x_0,r)$, $x_0 \in \R^n$ and $r>0$, fix $p$, $1\leq p<\infty$. 
\begin{enumerate}
    \item If $1\leq p<n$, choose $q$, $q\leq p<n$, choose $\sigma$, $1\leq \sigma \leq \frac{n}{n-q}$, and choose $\tau$, so that $1\leq \tau\leq\frac{\sigma q}{p}$.
    \item If $n\leq p<\infty$, choose $\tau$, $\tau\ge1$, choose $q$, $n\leq q\leq p$, and choose $\sigma$, $\sigma\ge1$, so that $1\leq \tau\leq \frac{\sigma q}{p}$.
\end{enumerate}
Given weights $v$ and $w$,     if $v\in L^{\frac{\sigma q}{\sigma q-\tau p}}(B)$ and $w^{-1} \in L^{\frac{q}{p-q}}(B)$, and if the matrix $Q$ satisfies the lower ellipticity condition~\eqref{ellptcy for gain}, then for any $u\in QH^{1,p}(v,\overline{B})$,
\begin{multline}\label{Ineq for loc. Poin}
    \bigg(\int_B |u-\langle u \rangle_{B,v}|^{\tau p} \,v dx\bigg)^\frac{1}{\tau p} \\
    \leq 2P(q,\sigma) r|B|^{\frac{1}{\sigma q}-\frac{1}{q}}\|w^{-1}\|^{\frac{1}{p}}_{L^{\frac{q}{p-q}}(B)}
    \|v\|^{\frac{1}{\tau p }}_{L^{\frac{\sigma q}{\sigma q-\tau p}}(B)}
    \bigg(\int_B|\sqrt{Q}\nabla u|^p \, dx\bigg)^{\frac{1}{p}},
\end{multline} 
where $P(q,\sigma)$ is the constant in inequality~\eqref{eqn:poincare-ball-alt}, and $\langle u \rangle_{B,v} = v(B)^{-1}\int_B u\,vdx$.
\end{theorem}

\begin{remark}
    In applications to the regularity of solutions of elliptic PDEs, an important feature of the classical Poincar\'e inequality is the factor $r$ in the constant on the right-hand side: i.e., the constant becomes arbitrarily small as $r$ tends to $0$. We can choose the exponents $q$ and $\sigma$ so that the constant in~\eqref{Ineq for loc. Poin} can also be made small:  see Section~\ref{section:pdes} below.
\end{remark}

\begin{remark}
    In Theorem~\ref{local Poin.} we have  assumed that the balls are Euclidean balls.  However, if we have that the Poincar\'e inequality~\eqref{eqn:poincare-ball-alt} holds for balls defined with respect to any quasi-metric, then we can prove inequality~\eqref{Ineq for loc. Poin} for such balls as well.  
\end{remark}

\subsection*{Poincar\'e Inequalities for $s$-John Domains}

We now give a version of the Poincar\'e inequality  for John domains and their generalization, $s$-John domains.  A bounded domain $\Omega \subset \R^n$ is called a John domain if there is a constant $0<c_J \leq 1$ and a distinguished point $x_0\in \Omega$ (the central point) such that each point $x\in \Omega$ can be joined to $x_0$ by a rectifiable curve $\gamma:[0,l]\rightarrow \Omega$ such that $\gamma(0)=x, \gamma(l)=x_0$, parametrized by arc length,  and for every $t\in[0,l]$,
    \begin{equation} \label{eqn:john-domain}
        \dist(\gamma(t),\partial\Omega)\ge c_Jt.
    \end{equation}
    Examples of John domains include metric balls, Lipschitz domains, star-shaped domains, domains that satisfy the cone condition, the Koch snow flake, and quasidisks. (This condition is equivalent to satisfying the so-called Boman chain condition. For further information, see\cite{Jiang-Kauranen,Diening-Ruzicka-Schumacher,Chuaqui-Osgood-Pommerenke}).  
    %Given a bounded John domain $\Omega$, the following  Poincar\'e inequality holds:  given $u\in \lip(\overline{\Omega})$
%     %
% \begin{equation}\label{Poin. for John domain}
%     \bigg(\int_\Omega |u-u_\Omega|^{\sigma q} \, dx\bigg)^\frac{1}{\sigma q} \leq P(c_J,n,q) \bigg(\int_\Omega \big |\nabla u \big|^q\, dx\bigg)^\frac{1}{q} 
% \end{equation}
% %
% where $1\leq q<n$ and  $1\leq \sigma \leq\frac{n}{n-q}$. (See ~\cite[Theorem 8]{Hajlasz}; again, for $\sigma<\frac{n}{n-q}$ this follows from H\"older's inequality.)
%
More generally, a bounded domain $\Omega \in \R^n$ is called an $s$-John domain, $s\ge1$, if~\eqref{eqn:john-domain} holds with $t^s$ on the right-hand side instead of $t$.  Examples of $s$-John domains include so-called mushroom domains and domains with a $1/s$-H\"older continuous boundary.  (See~\cite{buckley-koskela,Hajlasz-Koskela}.)

Given an $s$-John domain $\Omega$, $s\geq1$, for any $q$, $1\leq q < n_s = (n-1)s+1$, and for any $\sigma$,  $1\leq\sigma< \frac{n}{n_s-q} =\frac{n}{(n-1)s+1-q}$, the following Poincar\'e inequality holds: given $u\in \lip(\overline{\Omega})$,
    \begin{equation}\label{Poin. for s-John domain}
    \bigg(\int_\Omega |u-\langle u\rangle_\Omega|^{\sigma q} \, dx\bigg)^\frac{1}{\sigma q} 
    \leq P(c_J,s,n,q) \bigg(\int_\Omega \big |\nabla u \big|^q\, dx\bigg)^\frac{1}{q}.
    \end{equation}
If $q\geq n_s$, then \eqref{Poin. for s-John domain} holds for any $\sigma \geq 1$.  (See~\cite[Corollary 6, Remark 13]{Hajlasz-Koskela}.)  Note that when $s=1$, $n_s=\frac{n}{n-q}$, so for John domains we get the same gain as in~\eqref{eqn:poincare-ball}.

\begin{theorem}\label{s-John}
    Fix $s\geq 1$ and let $\Omega$ be an $s$-John domain.  Fix $p$, $1<p<\infty$.
    \begin{enumerate}
        \item If $1\leq p<n_s$, choose $q$, $1\leq q\leq p <n_s$, choose $\sigma$,  $1\leq \sigma<\frac{n}{n_s-q}$, and choose $\tau$, so that $1\leq\tau\leq \frac{\sigma q}{p}$,
        \item if $n_s\leq p<\infty$, choose $\tau$, $\tau\ge1$, choose $q$, $n_s\leq q\leq p$, and choose $\sigma$, $\sigma\ge1$, so that $1\leq\tau\leq \frac{\sigma q}{p}$.
    \end{enumerate}
Given weights $v$ and $w$, if $v\in L^{\frac{\sigma q}{\sigma q-\tau p}}(\Omega)\,$, $ \,w^{-1} \in L^{\frac{q}{p-q}}(\Omega)$, and if the matrix $Q$ satisfies the lower ellipticity condition \eqref{ellptcy for gain}, then for any $u\in QH^{1,p}(v,\overline{\Omega})$, 
\begin{multline}\label{Poin. for s-John}
    \bigg(\int_\Omega |u-\langle u\rangle_{\Omega,v}|^{\tau p} \,v dx\bigg)^\frac{1}{\tau p} \\
    \leq 2P(c_J,s,n,q)\|w^{-1}\|^{\frac{1}{p}}_{L^{\frac{q}{p-q}}(\Omega)}\|v\|^{\frac{1}{\tau p}}_{L^{\frac{\sigma q}{\sigma q-\tau p}}(\Omega)}\bigg(\int_\Omega|\sqrt{Q}\nabla u|^p \, dx\bigg)^{\frac{1}{p}},
\end{multline}
where $P(c_J,s,n,q)$ is the constant in~\eqref{Poin. for s-John domain} and $\langle u\rangle_{\Omega,v}=v(\Omega)^{-1}\int_\Omega u\,vdx$.
\end{theorem}

\subsection*{Folland-Stein inequalities}
We now briefly consider degenerate Sobolev inequalities in other geometries.  To prove these results, we build upon Remark~\ref{remark:measure-space}, where we noted that our extrapolation results hold in any measure space and not just in $\R^n$ equipped with Lebesgue measure.  We first consider the Heisenberg group and then CR manifolds. To avoid making some technical definitions, we will not  state our results on the largest family of functions for which they might hold; this allows us to omit approximation arguments in our proofs.

\medskip

 We begin with some basic definitions.  For further details, we refer the reader to~\cite{MR1232192,Yang,Monticelli-Li}.   The Heisenberg group, $n\ge1$, is the space $\mathbb{H}^n=\R^n\times\R^n\times\R$ with the group action $\circ$ defined by
 \begin{equation*}
     \hat{\xi}\circ\xi:=\bigg(x+\hat{x}, y+\hat{y},t+\hat{t}+2\sum_{\alpha=1}^nx_\alpha\hat{y}_\alpha-y_\alpha\hat{x}_\alpha\bigg)
 \end{equation*}
 for any $\xi=(x,y,t)$, $\hat{\xi}=(\hat{x},\hat{y},\hat{t})$ in $\mathbb{H}^n$,  where $x=(x_1,...,x_n)$,  $y=(y_1,...,y_n)$, $\hat{y}=\hat{y}_1,...,\hat{y}_n$ and $x_i,\,y_i,\,t\in \R$, and $\hat{x},\,\hat{y},\,\hat{t}$ are defined similarly.  The vector fields
    \begin{equation*}
       X_\alpha=\frac{\partial}{\partial x_\alpha}+2y_\alpha\frac{\partial}{\partial t}, \qquad Y_\alpha=\frac{\partial}{\partial y_\alpha}-2x_\alpha\frac{\partial}{\partial t}, \qquad T=\frac{\partial}{\partial t},
   \end{equation*}
$1\leq \alpha \leq n$, form a basis for the Lie-algebra of left invariant vector fields on $\mathbb{H}^n$ with respect to the group action $\circ$.  The horizontal gradient $\nabla_H$ on $\mathbb{H}^n$ is the first order operator
\begin{equation*}
    \nabla_Hu:=(X_1u,..., X_nu,Y_1u,...,Y_nu),
\end{equation*}
where $u\in C^\infty(\mathbb{H}^n)$.  We note that the homogeneous dimension of $\mathbb{H}^n$ is $m=2n+2$, which is different from the topological dimension $2n+1$.

Our starting point for extrapolation is the Folland-Stein Sobolev inequality (see~\cite{Folland-Stein}).  Given any $1<q<m$, let $\sigma=\frac{m}{m-q}$;  then  
\begin{equation}\label{Sob. for Heisenberg group}
    \bigg(\int_{\mathbb{H}^n} |u|^{\sigma q} \, dV\bigg)^\frac{1}{\sigma q} 
    \leq H(n,q) \bigg(\int_{\mathbb{H}^n} \big |\nabla_H u \big|^q\, dV\bigg)^\frac{1}{q}    
   \end{equation}
   holds for any $u\in C^\infty_0(\mathbb{H}^n)$, where $dV$ is the volume form on $\mathbb{H}^n$,
   $dV=dtdx_1...dx_ndy_1...dy_n$.

\begin{theorem}\label{Heisenberg Group} 
On the Heisenberg group $\mathbb{H}^n$, fix any $p$, $1\leq p<m$. Choose $q$, $1\leq q\leq p$, let $\sigma=\frac{m}{m-q}$, and  choose $\tau$, $1\leq\tau\leq\sigma$, so that $p\leq \frac{\sigma q}{\tau}$.  Given weights $v$ and $w$, if $v\in L^{\frac{\sigma q}{\sigma q-\tau p}}(\mathbb{H}^n)$ and $w^{-1} \in L^{\frac{q}{p-q}}(\mathbb{H}^n)$,  and the matrix $Q$ satisfies the lower ellipticity condition~\eqref{ellptcy for gain}, then for any $u\in C^\infty_0(\mathbb{H}^n)$,
\begin{equation*}\label{Sob. for Heis. Group}
    \bigg(\int_{\mathbb{H}^n} |u|^{\tau p} v\, dV\bigg)^\frac{1}{\tau p}
    \leq H(n,q)\|w^{-1}\|^{\frac{1}{p}}_{L^{\frac{q}{p-q}}(\mathbb{H}^n)}
    \|v\|^{\frac{1}{\tau p}}_{L^{\frac{\sigma q}{\sigma q-\tau p}}(\mathbb{H}^n)}
    \bigg(\int_{\mathbb{H}^n}|\sqrt{Q}\nabla_H u|^p \, dV\bigg)^{\frac{1}{p}},
    \end{equation*}
where $H(n,q)$ is the constant in~\eqref{Sob. for Heisenberg group}.
\end{theorem}

\medskip

We now consider Folland-Stein type inequalities on compact Riemannian manifolds and compact CR manifolds.  For brevity, we omit many technical definitions and refer the reader to~\cite{Aubin1976,Folland-Stein,Ho} for complete details.  

 Let $(M,g)$ be a compact Riemannian manifold  of dimension $n>2$ with constant curvature. Given a $C^\infty(M)$ function $u$, let $\nabla_g(u)$ be the Riemannian gradient of $u$. 
 %
 %In local coordinates defined by $\varphi:M\rightarrow \mathbb{R}^n$ with $\hat{x} = u(x)$, $\nabla_g u$ satisfies 
% %
% \[\widehat{ \nabla_g u} (\hat{x}) = G_{x}^{-1}\nabla \hat{u}(\hat{x})\]
% %
% where $G_x$ is the matrix whose entries are $g_{ij} = g(E_i,E_j)$ and the $E_j$'s are the coordinate vector-fields at $x\in M$.
In this setting the following Folland-Stein Sobolev type inequality holds (see\cite[Theorem 11]{Aubin1976}): given any $1\leq q<n$, let $\sigma = \frac{n}{n-q}$.  Then for every $u\in C^\infty(M)$,
\begin{equation} \label{eqn:FS-Rmanifold}
\bigg( \int_M |u|^{\sigma q}\, dV_g\bigg)^{\frac{1}{\sigma q}} 
\leq K(n,q) \bigg( \int_M |\nabla_g u|^{ q}\, dV_g\bigg)^{\frac{1}{ q}} 
+ A(q)\bigg( \int_M |u|^{ q}\, dV_g\bigg)^{\frac{1}{ q}}.
\end{equation}
Here, $dV_g$ is the volume form on $M$.

\begin{theorem}\label{riemannian} 
Let  $(M,g)$ be a compact Riemannian manifold of dimension $n>2$ with constant curvature.  Fix any $p$, $1\leq p <n$.  Choose $q$, $1\leq q \leq p<n$, let $\sigma  = \frac{n}{n-q}$, and choose $\tau$, so that $1\leq \tau \leq \frac{\sigma q}{p}$.    Given weights $v$ and $w$,   if $v\in L^{\frac{\sigma q}{\sigma q-\tau p}}(M)$ and $w^{-1} \in L^{\frac{q}{p-q}}(M)$, and the matrix $Q$ satisfies the lower ellipticity condition~\eqref{ellptcy for gain},  then for every $u\in C^\infty(M)$,
\begin{multline*} \bigg( \int_M |u|^{\sigma q}\, vdV_g\bigg)^{\frac{1}{\sigma q}} 
\leq \|w^{-1}\|^{\frac{1}{p}}_{L^{\frac{q}{p-q}}(M)}
    \|v\|^{\frac{1}{\tau p }}_{L^{\frac{\sigma q}{\sigma q-\tau p}}(M)}  \\
   \times  \bigg[K(n,q) \bigg( \int_M |\sqrt{Q}\nabla_g u|^{ q}\, dV_g\bigg)^{\frac{1}{ q}} 
+ A(q)\bigg( \int_M |u|^{ q}\, wdV_g\bigg)^{\frac{1}{ q}}\bigg],
\end{multline*}
 where $A(q),$ $K(n,q)$ are the constants in \eqref{eqn:FS-Rmanifold}.
\end{theorem}

\medskip

The corresponding theorem holds for CR manifolds.  
% Let $M$ be a real oriented $C^\infty$ manifold of dimension $2n+1$, $(n=1,2,3...)$ and $T^{1,0}$ be a subbundle of the complex tangent bundle $\C TM$. We will call $M$ a CR manifold if it satisfies each of the following.
%    \begin{enumerate}
%        \item $\dim^\C T^{1,0}=n$,
%        \item $T^{1,0}\cap T^{0,1}=\{0\}$,
%        \item $T^{1,0}$ is integrable in the sense of Frobenius,
%        \item $(T^{1,0}\oplus T^{0,1})^\perp \subset \C T^*M$ has a global section.
%    \end{enumerate}
%    See for detailed information ~\cite[Section 2]{Folland-Stein}.
   Given a contact form $\theta$, let $(M,\theta)$ be a compact, strictly pseudoconvex CR manifold of real dimension $l=2n+1$. For $u\in C^\infty(M)$, let $\nabla_b u$ denote the subgradient of $u$.  The following Folland-Stein Sobolev inequality holds (see~\cite{Folland-Stein,Ho}): given any $1<q<l$, let $\sigma=\frac{l}{l-q}$; then for every $u\in C^\infty(M)$,
   \begin{equation}\label{Ineq. Heisen. 2}
       \bigg(\int_M |u|^{\sigma q}\, dV_\theta\bigg)^\frac{1}{\sigma q}
       \leq L(l,q)  \bigg(\int_M |\nabla_bu|^{q}\, dV_\theta\bigg)^\frac{1}{q}
       + B(l,q)  \bigg(\int_M |u|^{q}\, dV_\theta\bigg)^\frac{1}{q}.  
   \end{equation}
Here, $dV_\theta=\theta\wedge (d\theta)^n$ is the volume form on $M$.

\begin{theorem}\label{Heis. Gr. Cor. 2}
Let $(M,\theta)$ be a compact, strictly pseudoconvex CR manifold of real dimension $l=2n+1$.
Fix any $p$, $1\leq p<l$. Choose $q$, $1\leq q\leq p$, let $\sigma=\frac{l}{l-q}$, and choose $\tau$, $1\leq\tau\leq\sigma$, so that $1\leq \tau \leq \frac{\sigma q}{p}$.  Given weights $v$ and $w$, if  $v\in L^{\frac{\sigma q}{\sigma q-\tau p}}(M)$ and  $w^{-1} \in L^{\frac{q}{p-q}}(M)$,  and the matrix  $Q$ satisfies the lower ellipticity condition \eqref{ellptcy for gain}, then for every $u\in C^\infty(M)$,
   \begin{multline*}
  \bigg(\int_M |u|^{\tau p} \,v dV_{\theta}\bigg)^\frac{1}{\tau p}
  \leq \|w^{-1}\|^{\frac{1}{p}}_{L^{\frac{q}{p-q}}(M)}
  \|v\|^{\frac{1}{\tau p}}_{L^{\frac{\sigma q}{\sigma q-\tau p}}(M)} \\
  \times \bigg[L(l,q)\bigg(\int_M |\sqrt{Q}\nabla_b u|^p \, dV_\theta\bigg)^{\frac{1}{p}}
  +B(l,q)\bigg(\int_M |u|^p\,w dV_\theta\bigg)^\frac{1}{p}\bigg], 
\end{multline*}
 where $L(l,q)$ and $B(l,q)$ are the constants in~\eqref{Ineq. Heisen. 2}.
\end{theorem}

\section{Applications to PDEs}
\label{section:pdes}

In this section we apply the results in Section~\ref{section:Applications} to the study of degenerate elliptic equations.  As we briefly discussed in the Introduction, the authors, with \c{C}etin and Zeren~\cite{SC-DCU-FED-SR-Zeren} proved the existence and uniqueness of weak solutions of  the Dirichlet problem
\begin{equation} \label{eqn:intro-dirichlet}
\begin{cases}
Lu = f + v^{-1}\Div(v \vece h), & x \in \Omega, \\
 u = 0, & x \in \partial \Omega,
 \end{cases}
\end{equation}
where $\Omega\subset \R^n$ is a bounded domain and $L$ is the degenerate, second order elliptic operator with lower order terms,
\begin{equation} \label{eqn:intro-operator}
 Lu = v^{-1} \Div(Q\nabla u) + \vecb \cdot \nabla u + v^{-1}\Div(v\vecc u) + du,
 \end{equation}
and where $Q$ satisfies the ellipticity condition~\eqref{elpt con.} with the weight $v$ in the equation and lower eigenvalue $w$ of $Q$.  To prove their results they assumed that there exists a global Sobolev inequality either with gain in the scale of Lebesgue spaces~\eqref{eqn:gen-sobolev-intro} or in the scale of Orlicz spaces~\eqref{eqn:gen-sobolev-orlicz-intro}, and that there exists a weak local Poincar\'e inequality of the form~\eqref{eqn:weak-poincare-intro}.  Similarly, the first and third authors, with MacDonald, studied the boundedness of solutions of the Dirichlet problem for the equation $Lu = -v^{-1}\Div(\sqrt{Q}\nabla u|^{p-2}Q\nabla u) = |f|^{p-2}f$, $1<p<\infty$.  Again, they assumed the existence of 
either inequality~\eqref{eqn:gen-sobolev-intro} or~\eqref{eqn:gen-sobolev-orlicz-intro} in their theorems (with $2$ replaced by $p$).  

In the context of these results, a natural question to ask is: when do such inequalities exist?  The approach of Fabes, Kenig and Serapioni~\cite{Fabes-Kenig-Serapioni} and Chanillo and Wheeden~\cite{Chanillo-Wheeden} was to prove the existence of Sobolev and Poincar\'e inequalities by assuming strong structural conditions on the weights $v,\,w$: doubling and the Muckenhoupt $A_p$ condition.  But by Theorems~\ref{global sob.} and~\ref{Orlicz case cor}, we can explicitly construct such inequalities by choosing $v$ and $w$ with the appropriate integrability, and the matrix $Q$ so that $v$ and $w$ are its largest and smallest eigenvalues:  for instance, we could take $Q$ to be the diagonal matrix $\diag(w,\lambda_1,...,\lambda_{n-2},v)$, where the functions $\lambda_i$ satisfy $w(x) \leq \lambda_i(x) \leq v(x)$.  

We can also show the existence of the weak Poincar\'e inequality.  The constant in the local Poincar\'e inequality in Theorem~\ref{local Poin.}, when $p=2$ and $\tau=1$,  can be written as
\[ 2P(q,\sigma) r^{1+ \frac{n}{\sigma q}-\frac{n}{q}}\|w^{-1}\|^{\frac{1}{2}}_{L^{\frac{q}{2-q}}(B)}
    \|v\|^{\frac{1}{ 2 }}_{L^{\frac{\sigma q}{\sigma q-2}}(B)}, \]
where $1\leq q\leq 2$ and $1\leq \sigma \leq \frac{n}{n-2}$ satisfy $\sigma q\geq 2$.  If we let $\sigma =\frac{n}{n-2}$, then we can take any $q$ such that $\frac{2n}{n+2} \leq q < 2$.  In this case, we have that the exponent $1+ \frac{n}{\sigma q}-\frac{n}{q}=0$.  If we assume the stronger condition that $v\in L^{\frac{\sigma q}{\sigma q-2\tau }}(\Omega)$ and $w^{-1} \in L^{\frac{q}{2-q}}(\Omega)$ for some $\tau>1$ then Theorem~\ref{global sob.} produces inequality~\eqref{eqn:gen-sobolev-intro} (which is also needed for these results).  The continuity of the integral then gives $\|w^{-1}\|^{\frac{1}{2}}_{L^{\frac{q}{2-q}}(B)}
    \|v\|^{\frac{1}{ 2 }}_{L^{\frac{\sigma q}{\sigma q-2}}(B)}\to 0$ as $r\to 0$ uniformly for all balls $B\subset \Omega$.  If we restrict $1<\sigma <\frac{n}{n-2}$, then we can find $q<2$ such that $\sigma q\geq 2$ and $1+ \frac{n}{\sigma q}-\frac{n}{q}=\nu>0$.  Thus we can get a weak Poincar\'e inequality of the form~\eqref{eqn:alt-poincare-intro} whenever $1<\sigma\leq \frac{n}{n-2}$.   Similar observations hold if we take $p\neq 2$ and if we choose $v$ and $w$ so that the Sobolev inequality with gain in Orlicz spaces~\eqref{eqn:gen-sobolev-orlicz-intro} holds.

\section{Proof of extrapolation theorems}
\label{section:proof-extrapol}

In this section we prove our extrapolation theorems.  We recall our convention that the space $L^{\frac{a}{b}}(\Omega)$, if $b=0$, is interpreted as $L^\infty(\Omega)$, and we make a similar convention for Orlicz space norms.  

\subsection*{Proof of Theorem \ref{extrapolation} }\label{section:Proof of Thm 2.1}
Fix a pair $(f,F)\in \F$.  Then by H\"older's inequality with exponents $\frac{\sigma q}{\tau p}$ and $\frac{\sigma q}{\sigma q- \tau p}$ (if $\sigma q - \tau p>0$, and by the definition of the $L^\infty$ otherwise), and by our assumption~\eqref{Ineq with sigma q in gen.}, 
\begin{multline}\label{eqn:extrapol1}
  \bigg(\int_\Omega |f|^{\tau p} \,v dx\bigg)^\frac{1}{\tau p} 
  \leq \bigg(\|v\|_{L^\frac{\sigma q}{\sigma q-\tau p}(\Omega)}
  \bigg(\int_\Omega |f|^{\sigma q} \,dx\bigg )^{\frac{\tau p}{\sigma q}}\bigg)^\frac{1}{\tau p}\\
  = \|v\|^\frac{1}{\tau p}_{L^\frac{\sigma q}{\sigma q-\tau p}(\Omega)}\bigg(\int_\Omega |f|^{\sigma q}\, \,dx\bigg )^{\frac{1}{\sigma q}} 
  \leq c_0\|v\|^\frac{1}{\tau p}_{L^\frac{\sigma q}{\sigma q-\tau p}(\Omega)}\bigg(\int_\Omega |F|^q\, \,dx\bigg )^{\frac{1}{q}}.
\end{multline} 

If we let $\xi=F(x)$ in \eqref{ellptcy for gain} and rearrange terms, we get that for a.e.~$x\in \Omega$,
\begin{align*}
  |F(x)|^q\leq  w(x)^{-\frac{q}{p}} \big|\sqrt{Q(x)}F(x) \big|^q.
\end{align*}
Hence, if we substitute this into the last term above and apply H\"older's inequality 
 with exponents $\frac{p}{q}$ and $\frac{p}{p-q}$ (again if $p>q$), we get that
\begin{align*}
      \bigg(\int_\Omega |f|^{\tau p} \,v dx\bigg)^\frac{1}{\tau p}
      &\leq c_0\|v\|^\frac{1}{\tau p}_{L^\frac{\sigma q}{\sigma q-\tau p}(\Omega)}\bigg( \int_\Omega w^{-\frac{q}{p}} |\sqrt{Q}F|^q \, dx\bigg)^{\frac{1}{q}} \\
    &\leq c_0 \|v\|^\frac{1}{\tau p}_{L^\frac{\sigma q}{\sigma q-\tau p}(\Omega)}\bigg[\bigg( \int_\Omega w^{\frac{-q}{p-q}} \, dx\bigg)^{\frac{p-q}{p}}  \bigg(\int_\Omega|\sqrt{Q}F|^p \, dx\bigg)^\frac{q}{p}\bigg]^\frac{1}{q}
    \\
    &=c_0\|v\|^\frac{1}{\tau p}_{L^{\frac{\sigma q}{\sigma q-\tau p}}(\Omega)}\|w^{-1}\|^{\frac{1}{p}}_{L^{\frac{q}{p-q}}(\Omega)}\bigg(\int_\Omega|\sqrt{Q}F|^p \, dx\bigg)^\frac{1}{p}.
\end{align*}
This proves inequality~\eqref{Ineq for gen.}.  If $p=q$, this inequality follows from immediately from the definition of $L^\infty$.  This completes the proof.

\subsection*{Proof of Theorem \ref{Main Thm 3}}\label{Weak Ver. Proof}
     The proof of Theorem \ref{Main Thm 3} is nearly identical to the proof of Theorem~\ref{extrapolation} so we will only sketch the changes.
     Again fix $(f,F)\in \F$.  If we argue as for inequality~\eqref{eqn:extrapol1} but use the assumption~\eqref{Ineq with sigma q for MT3}, we get
     \begin{equation*}
    \bigg(\int_\Omega |f|^{\tau p} \,v dx\bigg)^\frac{1}{\tau p}
 \leq \|v\|^\frac{1}{\tau p}_{L^\frac{\sigma q}{\sigma q-\tau p}(\Omega)}
   \bigg[ \bigg(c_0\int_\Omega |F|^q \,dx\bigg)^{\frac{1}{q}} + c_1\bigg(\int_\Omega |f|^q\,dx\bigg)^{\frac{1}{q}}\bigg].
\end{equation*}
The estimate for the first term on the right-hand side is identical to the one in the proof of Theorem~\ref{extrapolation}, so we only consider the estimate for the second term.  However, this is straightforward:  by H\"older's inequality 
 with exponents $\frac{p}{q}$ and $\frac{p}{p-q}$ (again if $p>q$), we get that
 \begin{multline*} \bigg(\int_\Omega |f|^q\,dx\bigg)^{\frac{1}{q}} 
 = \bigg(\int_\Omega |f|^q w^{\frac{q}{p}}w^{-\frac{q}{p}}\,dx\bigg)^{\frac{1}{q}} \\
 \leq \bigg( \int_\Omega w^{\frac{-q}{p-q}} \, dx\bigg)^{\frac{p-q}{p}}\bigg(\int_\Omega |f|^p\,wdx\bigg)^{\frac{1}{p}} 
 = \|w^{-1}\|^{\frac{1}{p}}_{L^{\frac{q}{p-q}}}\bigg(\int_\Omega |f|^p\,wdx\bigg)^{\frac{1}{p}}.
\end{multline*}
Inequality~\eqref{Ineq for MT3} follows.  If $p=q$ the same estimate is immediate.  This completes the proof.

\subsection*{Proof of Theorem \ref{Orlicz case}}\label{Orlicz proof}
    Fix $(f,F)\in \F$.  Define $q^*=\sigma q >p\geq 1$.  Then by the duality inequality~\eqref{eqn:orlicz-dual}, H\"older's inequality with exponents $q^*$ and $(q^*)^{'}=\frac{\sigma q}{\sigma q-1}$, and inequality~\eqref{Ineq with sigma q in gen.}, there exists $g\in L^{\Bar{A}}(v,\Omega)$, $\|g\|_{L^{\Bar{A}}(v,\Omega)}=1$, such that
    \begin{multline} \label{eqn:orlicz1}
        \|f\|_{L^A(v,\Omega)}
        \leq 2 \int_\Omega fgv \,dx\leq 2\bigg(\int_\Omega |f|^{\sigma q}\, \,dx\bigg )^{\frac{1}{\sigma q}}
        \bigg(\int_\Omega |gv|^{\frac{\sigma q}{\sigma q-1}}\, \,dx\bigg )^{\frac{\sigma q-1}{\sigma q}} \\
  \leq 2c_0\bigg(\int_\Omega |F|^{q} \,dx\bigg )^{\frac{1}{ q}}
  \bigg(\int_\Omega |g|^{\frac{\sigma q}{\sigma q-1}} v^{\frac{1}{\sigma q-1}} v \,dx\bigg )^{\frac{\sigma q-1}{\sigma q}}.
\end{multline}

We estimate the first integral in the last line exactly as we did in the proof of Theorem~\ref{extrapolation} to get:
\begin{equation} \label{eqn:orlicz2} \bigg(\int_\Omega |F|^{q} \,dx\bigg )^{\frac{1}{ q}}
\leq \|w^{-1}\|^{\frac{1}{p}}_{L^\infty(\Omega)}\bigg(\int_\Omega|\sqrt{Q}F|^p \, dx\bigg)^{\frac{1}{p}}.
\end{equation}
Therefore, to complete the proof we need to estimate the second integral.  We will apply H\"older's inequality in the scale of Orlicz spaces \eqref{Hölders Ineq for Orlicz} with
\[ B(t)=t^{\big(\frac{p'}{(q^*)'}\big)'}\log(e+t)^{\frac{\tau(p'-1)}{(p'/(q^*)')-1}} 
\qquad \text{and} \qquad 
\Bar{B}(t)\approx \frac{t^{\frac{p'}{(q^*)'}}}{\log(e+t)^{\tau(p'-1)}}; \]
note that since $p<q^*$, $p'>(q^*)'$, and so $B$ and $\Bar{B}$ are well-defined Orlicz functions.  By inequality~\eqref{eqn:conjugate-young}, $\Bar{A}(t)\approx\frac{t^{p'}}{\log(e+t)^{\tau(p'-1)}}$; thus  $\Bar{B}(t^{\frac{\sigma q}{\sigma q-1}})=\Bar{A}(t)$.  Similarly, if we define  $\Phi(t)\approx t^{\frac{p}{\sigma q-p}}\log(e+t)^{\frac{\tau\sigma q}{\sigma q-p}}$, then $B(t^{\frac{1}{\sigma q-1}})=\Phi(t)$.
Hence, by H\"older's inequality and the rescaling property~\eqref{rescaling},  
\begin{multline*}
      \bigg(\int_\Omega |g|^{\frac{\sigma q}{\sigma q-1}} v^{\frac{1}{\sigma q-1}} v \,dx\bigg )^{\frac{\sigma q-1}{\sigma q}}
      \leq 2\|g^{\frac{\sigma q}{\sigma q-1}}\|^{\frac{\sigma q-1}{\sigma q}}_{L^{\Bar{B}}(v,\Omega)}\|v^{\frac{1}{\sigma q-1}}\|^{\frac{\sigma q-1}{\sigma q}}_{L^B(v,\Omega)} \\
      = 2\|g\|_{L^{\Bar{A}}(v,\Omega)}\|v\|^{\frac{1}{\sigma q}}_{L^\Phi(v,\Omega)}
      = 2\|v\|^{\frac{1}{\sigma q}}_{L^\Phi(v,\Omega)}.
\end{multline*}
If we combine this estimate with inequalities~\eqref{eqn:orlicz1} and~\eqref{eqn:orlicz2}, to complete the proof of inequality~\eqref{Sob. Ineq Orlicz} we need to show that 
\begin{equation*}
    \|v\|_{L^\Phi(v,\Omega)}\leq \|v\|_{L^\Psi(\Omega)}(1+\|v\|_{L^\Psi(\Omega)}).
\end{equation*}

To see this, let $\Tilde{v}=v/\|v\|_{L^\Psi(\Omega)}$.
By inequality~\eqref{Orl Norm Rel.} and if we fix $\lambda=1$ in the definition of the Amemiya norm,
\begin{align*}
\|v\|_{L^\Phi(v,\Omega)}&=\|v\|_{L^\Psi(\Omega)}\|\Tilde{v}\|_{L^\Phi(v,\Omega)}\\
&\leq \|v\|_{L^\Psi(\Omega)}\|\Tilde{v}\|^*_{L^\Phi(v,\Omega)} \\
& = \|v\|_{L^\Psi(\Omega)} \inf\bigg\{ \lambda + \lambda \int_\Omega \Phi\bigg(\frac{|f|}{\lambda}\bigg)\,vdx : \lambda>0 \bigg\} \\
&\leq\|v\|_{L^\Psi(\Omega)}\bigg(1+\int_\Omega \Phi(\Tilde{v})\,v dx\bigg);\\
\intertext{since $\Psi(t)=t\Phi(t)$, we have that  }
  &=\|v\|_{L^\Psi(\Omega)}\bigg(1+\|v\|_{L^\Psi(\Omega)}\int_\Omega \Psi(\Tilde{v})\,dx\bigg) \\
  & \leq \|v\|_{L^\Psi(\Omega)}(1+\|v\|_{L^\Psi(\Omega)}),
\end{align*}
where the last inequality follows from~\eqref{eqn:modular} since $\|\Tilde{v}\|_{L^\Psi(\Omega)}=1$.  This completes the proof.

\section{Proof of Applications}\label{section:Proof of App.}
\label{section:proof:appl}

In this section we give the proofs of the applications stated in Section~\ref{section:Applications}.  All of the proofs are variations on the same theme:  as outlined in Remark~\ref{remark:two-step}, we will first construct a family of extrapolation pairs $\F$ such that a classical Sobolev or Poincar\'e holds, and for which the left-hand side of the corresponding degenerate inequality is finite.  Once we do this, the degenerate inequality for pairs in $\F$ follows at once by applying one of our extrapolation theorems. To complete the proof, we give an approximation argument that extends the degenerate inequality to all elements of the desired degenerate Sobolev space. 

\subsection*{Sobolev inequalities}
We begin with the proofs of our degenerate Sobolev inequalities.  We will give the proof of Theorem~\ref{global sob.} in careful detail so as to provide a model for applying extrapolation.  For Theorems~\ref{classical Sob.} and~\ref{Orlicz case cor} we will only sketch the important changes. 

\medskip

\begin{proof}[Proof of Theorem~\ref{global sob.}]
Fix the domain $\Omega$ and the values $p$, $q$, $\sigma$, and $\tau$ as in one of the three cases in the hypotheses of the theorem.  Then in every such case, we have that there exists a constant $S(q,\sigma)$ such that the Sobolev inequality~\eqref{classical Sob.},
 \begin{equation} \label{eqn:class-sobolev2}
        \bigg(\int_\Omega |u|^{\sigma q} \, dx\bigg)^\frac{1}{\sigma q} 
        \leq S(q,\sigma) \bigg(\int_\Omega \big |\nabla u \big|^q\, dx\bigg)^\frac{1}{q},
 \end{equation}
holds for every $u\in\lip_0(\Omega)$. Thus, the initial extrapolation inequality \eqref{Ineq with sigma q in gen.} holds for the  family of extrapolation pairs $\mathcal{F}=\{(u,\nabla u): u\in\lip_0(\Omega)\}$. Fix weights $v$ and $w$, and a matrix $Q$ satisfying the hypotheses.
Given any $u \in \lip_0(\Omega)$, $\supp(u)\subset \Omega$ is compact.  Since $v\in L^{\frac{\sigma q}{\sigma q-\tau p}}(\Omega) \subset L^1(\supp(u))$, we have that $u \in L^{\tau p}(v,\Omega)$.  Therefore, 
by Theorem~\ref{extrapolation},
 \begin{equation} \label{eqn:mid-step}
    \bigg(\int_\Omega |u|^{\tau p} \,v dx\bigg)^\frac{1}{\tau p}
    \leq S(q,\sigma)\|w^{-1}\|^{\frac{1}{p}}_{L^{\frac{q}{p-q}}(\Omega)}\|v\|^{\frac{1}{\tau p }}_{L^{\frac{\sigma q}{\sigma q-\tau p}}(\Omega)}\bigg(\int_\Omega|\sqrt{Q}\nabla u|^p \, dx\bigg)^{\frac{1}{p}}
\end{equation}
holds for any $u\in \lip_0(\Omega)$.

 To complete the proof we use a standard approximation argument.  Fix  $u\in QH^{1,p}_0(v,\Omega)$; then there is a Cauchy sequence $\{u_k\}_{k=1}^\infty$ in $Q\lip_0(v,\Omega)$ such that $\|u_k-u\|_{L^p(v,\Omega)}\rightarrow0$ and $\|\nabla u_k-\nabla u\|_{QL^p(\Omega)}\rightarrow0$ as $k\rightarrow\infty$. By passing to a subsequence we may assume that $u_k\to u$ pointwise $v$-almost everywhere.   Moreover, the pairs $(u_k,\nabla u_k) \in \F$.  Therefore, by  Fatou's lemma and inequality~\eqref{eqn:mid-step},
\begin{align*}
  \bigg( \int_\Omega |u|^{\tau p}\,vdx \bigg)^{\frac{1}{\tau p}}
  &\leq \mathop{\lim }\limits_{k\to \infty}  \bigg( \int_\Omega |u_k|^{\tau p}\,vdx \bigg)^{\frac{1}{\tau p}}\\
  &\leq S(q,\sigma)\|w^{-1}\|^{\frac{1}{p}}_{L^{\frac{q}{p-q}}(\Omega)}
  \|v\|^{\frac{1}{\tau p }}_{L^\frac{\sigma q}{\sigma q-\tau p}(\Omega)}
  \mathop{\lim }\limits_{k\to \infty}
 \bigg(\int_\Omega|\sqrt{Q}\nabla u_k|^p \, dx\bigg)^{\frac{1}{p}}\\
  &= S(q,\sigma)\|w^{-1}\|^{\frac{1}{p}}_{L^{\frac{q}{p-q}}(\Omega)}
  \|v\|^{\frac{1}{\tau p }}_{L^\frac{\sigma q}{\sigma q-\tau p}(\Omega)}
  \bigg(\int_\Omega|\sqrt{Q}\nabla u|^p \, dx\bigg)^{\frac{1}{p}}.
\end{align*}
This completes the proof. 
\end{proof}

\begin{proof}[Proof of Theorem~\ref{classical Sob.}]
  This result is actually a special case of Theorem~\ref{global sob.}, but it is just as easy to prove it directly.  Fix $1<p<\infty$ and let $q=\frac{np}{n+p-1}$. Since $n\ge2$ and $p>1$, $1<q<n$. Therefore, the classical Sobolev inequality~\eqref{eqn:class-sobolev2} holds
   for every $u\in\lip_0(\Omega)$ with  $\sigma=\frac{n}{n-q}$. Thus, extrapolation inequality \eqref{Ineq with sigma q in gen.} holds for 
$\mathcal{F}=\{(u,\nabla u): u\in \lip_0(\Omega)\}$.  A straightforward calculation shows that $\frac{p}{\sigma}=\frac{p(n-1)}{n+p-1}<q<p$, so \eqref{prop for 4.1} holds.  The lower ellipticity condition \eqref{elp.cond.} can be rewritten as
   \begin{equation*}
       \frac{1}{v^{p-1}}|\xi|^{p}\leq |\sqrt{Q}\xi|^p;
   \end{equation*}
 thus,~\eqref{ellptcy for gain} holds with $w=\frac{1}{v^{p-1}}$.   Finally, set $\tau = 1$.

 To apply Theorem~\ref{extrapolation}, we need to show that $v$ and $w$ satisfy the the integrability assumptions.  But another calculation shows that $\frac{q}{p-q}=\frac{n}{p-1}$ and $\frac{\sigma q}{\sigma q-p}=n$, so
 \[ \|w^{-1}\|^{\frac{1}{p}}_{L^{\frac{q}{p-q}}(\Omega)} 
 =\|v^{p-1}\|^{\frac{1}{p}}_{L^{\frac{q}{p-q}}(\Omega)} 
 = \|v\|^{\frac{1}{p'}}_{L^{n}(\Omega)}
 \quad \text{ and } \quad \|v\|^{\frac{1}{p}}_{L^{\frac{\sigma q}{\sigma q-p}}(\Omega)}= \|v\|^{\frac{1}{p}}_{L^{n}}.  
\]
 Therefore, by Theorem~\ref{extrapolation}, inequality~\eqref{deg. Sob.} holds for all $u\in \lip_0(\Omega)$ and the same approximation argument as in the proof of Theorem~\ref{global sob.} shows that it holds for all $u\in QH_0^{1,p}(v,\Omega)$. 
\end{proof}

\begin{proof}[Proof of Theorem~\ref{Orlicz case cor}]
The proof of this result is essentially the same as that of Theorem~\ref{global sob.}, except it uses Theorem~\ref{Orlicz case} instead of Theorem~\ref{extrapolation}.  With the given hypotheses, the classical Sobolev inequality~\eqref{eqn:class-sobolev2} holds for these values of $q$ and $\sigma$, and so with the same family of extrapolation pairs $\F$, Theorem~\ref{Orlicz case} yields inequality~\eqref{Orlicz ineq.} for $u\in \lip_0(\Omega)$.  Then the same approximation argument, this time using Fatou's lemma in the scale of Orlicz spaces (see~\cite[Section 3.2, Proposition 4]{Rao-Ren}) proves it for $u\in QH^{1,p}_0(v,\Omega$). 
\end{proof}

\subsection*{Poincar\'e inequalities}
Next, we prove Theorems~\ref{local Poin.} and~\ref{s-John}. Again, the proofs are very similar to the proof of Theorem~\ref{global sob.} and we focus on the necessary changes.

\begin{proof}[Proof of Theorem~\ref{local Poin.}]
With the given hypotheses we have that the local Poincar\'e inequality~\eqref{eqn:poincare-ball-alt} holds for $u\in \lip(\Bar{B})$.  Therefore, inequality~\eqref{Ineq with sigma q in gen.} holds for the  family of extrapolation pairs $\mathcal{F}=\{(u-\langle u \rangle_B,\nabla u): u\in\lip(\Bar{B})\}$, and so for weights $v$ and $w$ and matrix $Q$ as in the hypotheses, we have that
\begin{equation}\label{second part}
    \|u-\langle u \rangle _{B}\|_{L^{\tau p}(v,B)}
    \leq Cr|B|^{\frac{1}{\sigma q}-\frac{1}{q}}\|w^{-1}\|^{\frac{1}{p}}_{L^{\frac{q}{p-q}}(B)}\|v\|^{\frac{1}{\tau p }}_{L^{\frac{\sigma q}{\sigma q-\tau p}}(B)}\bigg(\int_B|\sqrt{Q}\nabla u|^p \, dx\bigg)^{\frac{1}{p}}.
\end{equation}
holds for all $u\in \lip(\Bar{B})$.  

Moreover, we claim that
\begin{equation}\label{first part}
    \|u-\langle u \rangle _{B,v}\|_{L^{\tau p}(v,B)}\lesssim 2\|u-\langle u \rangle _{B}\|_{L^{\tau p}(v,B)}
\end{equation}
which, combined with~\eqref{second part}, gives~\eqref{Ineq for loc. Poin} for all $u\in \lip(\Bar{B})$. To prove~\eqref{first part}, note that if $\tau p>1$, then by H\"older's inequality,
\begin{multline*}
    |\langle u \rangle_B-\langle u \rangle_{B,v}|
    =\bigg|\frac{1}{v(B)}\int_B(\langle u \rangle_B-u)\,v dx\bigg|
    \leq \frac{1}{v(B)}\int_B |u-\langle u \rangle_B|\,v^{\frac{1}{\tau p}}v^{\frac{\tau p-1}{\tau p}}\,dx \\
    \leq \frac{1}{v(B)}\bigg(\int_B|u-\langle u \rangle_B|^{\tau p}\,v dx\bigg)^\frac{1}{\tau p}
    \bigg(\int_B\,v dx\bigg)^\frac{\tau p-1}{\tau p}
    =v(B)^{\frac{-1}{\tau p}}\|u-\langle u\rangle_B\|_{L^{\tau p}(v,B)}.
\end{multline*}
When $\tau p =1$, this inequality is immediate.  If we take the norm,
\begin{multline*}
    \|\langle u \rangle_B-\langle u \rangle_{B,v}\|_{L^{\tau p}(v,B)}
    \leq\bigg(\int_B\big(v(B)^\frac{-1}{\tau p}\|u-\langle u \rangle _B\|_{L^{\tau p}(v,B)}\big)^{\tau p}\,v dx\bigg)^\frac{1}{\tau p}\\
    =v(B)^\frac{-1}{\tau p}\bigg(\int_B\,v dx\bigg)^\frac{1}{\tau p}\|u-\langle u \rangle_B\|_{L^{\tau p}(v,B)}
    =\|u-\langle u \rangle_B\|_{L^{\tau p}(v,B)}.
\end{multline*}
Given this, \eqref{first part} follows by the triangle inequality.

To complete the proof, we fix $u\in QH^{1,p}(v,B)$ and use the same approximation argument as before.  The only change is that given a Cauchy sequence $\{u_k\}_{k=1}^\infty$ in $\lip(\Bar{B})$ such that $\|u_k-u\|_{L^p(v,B)}\rightarrow 0$ and $\|\nabla u_k-\nabla u\|_{QL^p(B)}\rightarrow 0$ as $k\rightarrow\infty$, we also have that 
\begin{equation*}
    |\langle u_k \rangle_{B,v}-\langle u \rangle_{B,v}|
    \leq \frac{1}{v(B)}\int_B|u_k-u|\,v dx\leq v(B)^{-\frac{1}{p}} \|u_k-u\|_{L^p(v,B)}.
\end{equation*}
Thus, $\langle u_k \rangle_{B,v}\to \langle u \rangle_{B,v}$ as $k\to\infty$, and the approximation argument goes through.
This completes the proof.
\end{proof}

\begin{proof}[Proof of Theorem~\ref{s-John}]
The proof of this result is identical to the proof of Theorem~\ref{local Poin.}, with the only changes being we replace $B$ with $\Omega$,  use the Poincar\'e inequality~\eqref{Poin. for s-John domain} instead of the local Poincar\'e inequality~\eqref{eqn:poincare-ball-alt}, and change the constant that appears accordingly.
\end{proof}

\subsection*{Folland-Stein inequalities}
Finally, we sketch the proofs of Theorems~\ref{Heisenberg Group},~\ref{riemannian} and~\ref{Heis. Gr. Cor. 2}.

\begin{proof}[Proof of Theorem~\ref{Heisenberg Group}]
The proof is identical to the proof of Theorem~\ref{global sob.}, defining $\F =\{ u, \nabla_H u : u \in C^\infty_0(\mathbb{H})\}$, and omitting the approximation argument.  We can apply Theorem~\ref{extrapolation} in this setting by Remark~\ref{remark:measure-space}.
\end{proof}

\begin{proof}[Proof of Theorems~\ref{riemannian} and~\ref{Heis. Gr. Cor. 2}]
The proof of both of these Theorems are essentially the same as the proof of  Theorem~\ref{global sob.}, defining the extrapolation pairs using $\nabla_g u$ or $\nabla_b u$, $u\in C^\infty(M)$, applying Theorem~\ref{Main Thm 3} instead of Theorem~\ref{extrapolation}, and omitting the approximation argument.  We can apply Theorem~\ref{Main Thm 3} in this setting by Remark~\ref{remark:measure-space}.
\end{proof}

\bibliographystyle{plain}
\bibliography{bibliography2-DCU}

\end{document}